%% file: higher_q_cfs.tex
\documentclass[12pt]{amsart}

\usepackage{mathtools}
\def\mbinom#1#2{\ensuremath{\left(\kern-.3em\left(\genfrac{}{}{0pt}{}{#1}{#2}\right)\kern-.3em\right)}}

\usepackage[color,matrix,arrow]{xy}
\usepackage{graphicx}
\usepackage{color}
\usepackage{etoolbox}
\usepackage[margin=1in]{geometry}
\usepackage{amsmath,amsthm,amssymb,tikz,tikz-cd}
\usepackage{stmaryrd}
\usepackage{hyperref}
\hypersetup{
    colorlinks = true,
    citecolor = orange,%
}
\usepackage[noabbrev]{cleveref}
\usepackage{mathdots} %enables command \iddots for skewdiagonal dots 

\theoremstyle{plain}
\newtheorem{theorem}{Theorem}[section]

\newtheorem*{thmA}{Theorem A}
\newtheorem*{thmB}{Theorem B}
\newtheorem*{thmC}{Theorem C}
\newtheorem{lemma}[theorem]{Lemma}

\theoremstyle{definition}
\newtheorem{remark}[theorem]{Remark}
\newtheorem{example}[theorem]{Example}
\newtheorem{definition}[theorem]{Definition}
\input{tikzsetup.tex}

\makeatletter
\patchcmd{\@settitle}{\uppercasenonmath\@title}{}{}{}
\patchcmd{\@setauthors}{\MakeUppercase}{}{}{}
\patchcmd{\section}{\scshape}{}{}{}
\makeatother
\makeatletter
\patchcmd{\@maketitle}
  {\ifx\@empty\@dedicatory}
  {\ifx\@empty\@date \else {\vskip2ex %vertical position of date
  \centering\footnotesize\@date\par\vskip1ex}\fi
   \ifx\@empty\@dedicatory}
  {}{}
\patchcmd{\@adminfootnotes}
  {\ifx\@empty\@date\else \@footnotetext{\@setdate}\fi}
  {}{}{}
\makeatother

\newcommand{\calG}[0]{\mathcal{G}}
\newcommand{\wt}{\mathrm{wt}}

\title{Higher $q$-Continued Fractions}
\date{}

\author[A. Burcroff]{Amanda Burcroff$^\clubsuit$}
\thanks{$^\clubsuit$ \href{mailto:aburcroff@math.harvard.edu}{aburcroff@math.harvard.edu}, Harvard University, supported by the Jack Kent Cooke Foundation}
\author[N. Ovenhouse]{Nicholas Ovenhouse$^\heartsuit$}
\thanks{$^\heartsuit$\href{mailto:nicholas.ovenhouse@yale.edu}{nicholas.ovenhouse@yale.edu} Yale University, supported by Simons Foundation grant 327929}
\author[R. Schiffler]{Ralf Schiffler$^\spadesuit$}
\thanks{$^\spadesuit$\href{mailto:schiffler@math.uconn.edu}{schiffler@math.uconn.edu} University of Connecticut, supported by NSF grant DMS-2054561}
\author[S. W. Zhang]{Sylvester W. Zhang$^\diamondsuit$}
\thanks{$^\diamondsuit$\href{mailto:swzhang@umn.edu}{swzhang@umn.edu} University of Minnesota, supported by NSF grants DMS-1949896 \& DMS-1745638}

\begin{document}

\begin{abstract}
    We introduce a $q$-analog of the higher continued fractions introduced by the last three authors in a previous work (together with Gregg Musiker), which are 
    simultaneously a generalization of the $q$-rational numbers of Morier-Genoud and Ovsienko. They are defined
    as ratios of generating functions for $P$-partitions on certain posets. We give matrix formulas for computing
    them, which generalize previous results in the $q=1$ case. We also show that certain properties enjoyed by
    the $q$-rationals are also satisfied by our higher versions.
\end{abstract}

\maketitle

\tableofcontents

\section{Introduction}

The aim of the present paper is to define and study \emph{$q$-deformed higher continued fractions}. These are certain rational functions in $q$
which are simultaneous generalizations of two concepts introduced relatively recently. On the one hand, the last three authors (together with Gregg Musiker)
recently defined a higher-dimensional generalization of continued fractions called \emph{higher continued fractions} \cite{mosz_23}. 
Given an integer sequence $(a_1,a_2,\dots,a_n)$ the corresponding higher continued fraction
is a tuple of rational numbers defined by some nested recurrences which resemble the usual definition of continued fractions. Alternatively, they may be
defined as the ratio of the sizes of two combinatorially defined sets. There are several equivalent combinatorial descriptions of these sets, including
(chains of) order ideals of fence posets, $P$-partitions on fence posets, tuples of bounded lattice paths (up to some equivalence), 
and $m$-dimer covers on certain planar graphs (called \emph{snake graphs}). The main objects of study in the present work are $q$-analogs of these higher continued fractions,
where we replace the cardinality of these sets with their generating functions for some natural statistics.

On the other hand, Morier-Genoud and Ovsienko defined \emph{$q$-rational numbers} \cite{mgo_20}, which are some extension of the $q$-integers $[n]_q = 1+q+q^2 + \cdots + q^{n-1}$
to the case of $n \in \Bbb{Q}$, obtained as a $q$-deformation of the continued fraction expression (see Section \ref{sec:higher_q_cfs} for the definition).
Our higher continued fractions (and their $q$-analogs) depend on an index $m$, and when $m=1$ it specializes to the $q$-rational numbers. To strengthen this
analogy, we show that our higher $q$-continued fractions satisfy some of the nice properties of the $q$-rationals.

For the remainder of the introduction, we will give some more specific definitions and background, and give an overview of our main results.

Let $\lambda$ and $\mu$ be partitions (identified with their Young diagrams) with $\mu \leq \lambda$, and $\lambda / \mu$ the corresponding skew shape.
A \emph{reverse plane partition} on the shape $\lambda / \mu$ is a filling of the boxes of the Young diagram with non-negative integers
which is weakly increasing in rows (from left-to-right) and columns (top-to-bottom). Thinking of the set of boxes in the diagram as a poset
(where $x < y$ if box $x$ is weakly south-east of box $y$),
a reverse plane partition is equivalent to a \emph{$P$-partition} on this poset 
(and from now on we will use the terms ``plane partition'' and ``$P$-partition'' interchangeably). The numbers in the boxes are the \emph{parts} of the $P$-partition.

For a skew shape $P = \lambda/\mu$, we let $\Omega_m(P)$ be the set of $P$-partitions whose parts are at most $m$.
Also, we will let $\Omega_m(P,q)$ be the generating function $\sum_{\sigma \in \Omega_m(P)} q^{\mathrm{wt}(\sigma)}$, where
the \emph{weight} of a $P$-partition is the sum of its parts.
Plane partitions on Young diagrams and skew shapes are very well-studied (see \cite{gessel_16} for a survey).
Here, we will be concerned with the case when $\lambda / \mu$ is a \emph{border strip}; i.e. it contains no $2 \times 2$ block. Border strips are also
called \emph{ribbon shapes} or \emph{skew hooks}. The underlying graph of the skew shape is also sometimes called a \emph{snake graph} (see e.g. \cite{propp} and \cite{cs_18}).
The corresponding poset is sometimes called a \emph{fence poset} \cite{or_23}.

\begin {figure}
\centering
\begin {tikzpicture}[scale=0.7]
    \draw (0,0) -- (0,5);
    \draw (1,0) -- (1,5);
    \draw (2,4) -- (2,5);
    \draw (3,4) -- (3,7);
    \draw (4,4) -- (4,7);
    \draw (5,6) -- (5,7);
    \draw (6,6) -- (6,7);
    \draw (7,6) -- (7,7);

    \draw (0,0) -- (1,0);
    \draw (0,1) -- (1,1);
    \draw (0,2) -- (1,2);
    \draw (0,3) -- (1,3);
    \draw (0,4) -- (4,4);
    \draw (0,5) -- (4,5);
    \draw (3,6) -- (7,6);
    \draw (3,7) -- (7,7);
\end {tikzpicture}
\caption {The border strip (i.e. snake graph) $\mathcal{G}[5,3,2,4]$.}
\label {fig:snake_example}
\end {figure}
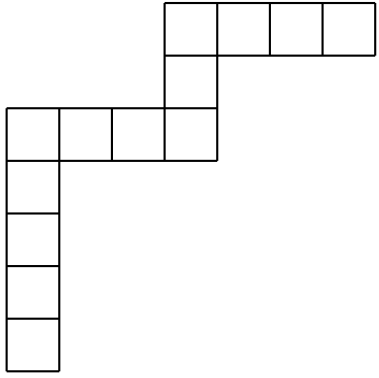

Let $P = \lambda/\mu$ be a border strip. 
Associated to $P$ is a sequence of positive integers $(a_1,a_2,\dots,a_n)$ describing its shape as follows. Each square of $P$ is attached to
the previous square's right or top edge, so we can describe the shape as a sequence of ``up'' and ``right'' steps.
Starting with the initial square, we have $(a_1-1)$ up steps, then $a_2$ right, then $a_3$ up, then $a_4$ right, etc, alternating. Finally,
we end with $(a_n-1)$ right steps (respectively up) if $n$ is even (respectively odd). See Figure \ref{fig:snake_example} for an example.
We will denote this skew shape (and the underlying graph) by $\mathcal{G}[a_1,\dots,a_n]$.

\begin {remark}
    We note that this is different from the construction in \cite{cs_18} for
    the snake graph of a continued fraction. What we call here $\mathcal{G}[a_1,\dots,a_n]$ is what \cite{propp} and \cite{claussen} called the \emph{dual snake graph}.
\end {remark}
Our first main result is the following.

\begin {thmA} (see Theorem \ref{thm:q_matrix_formula}) \\
    There exist matrices $R_m(q), L_m(q) \in \mathrm{GL}_{m+1}(\Bbb{Z}[q^\pm])$ such that for any sequence $(a_1,\dots,a_{2n})$, 
    and its corresponding snake graph $\mathcal{G} = \mathcal{G}[a_1,\dots,a_{2n}]$, the $(1,1)$-entry of the product
    \[X := R_m(q)^{a_1}L_m(q)^{a_2}R_m(q)^{a_3}L_m(q)^{a_4} \cdots R_m(q)^{a_{2n-1}}L_m(q)^{a_{2n}} \] 
    is equal to the $P$-partition generating function $\Omega_m(\mathcal{G}, q)$. 
    Moreover, every entry of this matrix is the generating function for some subset of $P$-partitions
    defined by restricting the values in the first and last box.
\end {thmA}

In \cite{mosz_23}, the authors defined \emph{higher continued fractions}, which are related to the enumeration of $P$-partitions on border strips. 
Given a sequence $(a_1,\dots,a_n)$, let $\mathcal{G} = \mathcal{G}[a_1,\dots,a_n]$, and $\mathcal{G}' = \mathcal{G}[1,a_2-1,a_3,\dots,a_n]$. In other words,
$\mathcal{G}'$ is obtained from $\mathcal{G}$ by deleting the first vertical column of boxes.
We associate to this sequence the following rational number, called the \emph{$m$-continued fraction}
\[ r_m(a_1,\dots,a_n) := \frac{\left|\Omega_m(\mathcal{G})\right|}{\left| \Omega_m(\mathcal{G}') \right|} \]
We also think of this as a map $\Bbb{Q} \to \Bbb{Q}$, where $r_m(x) := r_m(a_1,\dots,a_n)$ if $x \in \Bbb{Q}$ has continued fraction $[a_1,\dots,a_n]$.
More generally, there is a family of maps $r_{i,m}(x)$ for all $i \leq m$, where the one given above corresponds to the case $i=m$. The more general
case will be described in the next section.

We define \emph{$q$-deformed higher continued fractions} by replacing $\left| \Omega_m(\mathcal{G}) \right|$ with the polynomial $\Omega_m(\mathcal{G},q)$:
\[ r_m^q(a_1,\dots,a_n) := \frac{\Omega_m(\mathcal{G},q)}{\Omega_m(\mathcal{G}',q)} \]
By Theorem 4 in \cite{mgo_20}, when $m=1$ this is equal to Morier-Genoud and Ovsienko's $q$-rational number\footnote{The theorem in \cite{mgo_20} says the
numerator of this $q$-continued fraction is the generating function for \emph{closure subsets} of some directed graph, but it is easy to see that there is
a weight-preserving bijection between the set of closure subsets and the set of $P$-partitions with parts $0$ and $1$.}
\[ [a_1,\dots,a_n]_q := [a_1]_q + \cfrac{q^{a_1}}{[a_2]_{q^{-1}} + \cfrac{q^{-a_2}}{[a_3]_q + \cfrac{q^{a_3}}{[a_4]_{q^{-1}} + \cfrac{a^{-a_4}}{\ddots}}}} \]
One of the remarkable properties of the $q$-rationals is the ``stabilization phenomenon'' \cite{mgo_22}, which says that for a sequence $x_n$ approaching an irrational limit,
the power series expansions of the $q$-rationals stabilize to some well-defined power series, which we can take as the $q$-analog of the irrational number.
Our second main result is that the same is true in the higher case.

\begin {thmB}{(see Theorem \ref{thm:stabilization})} \\
    Let $x_n \in \Bbb{Q}$ be a sequence of rational numbers converging to an irrational number $x$. The power series expansions of $r^q_m(x_n)$ stabilize as $n \to \infty$,
    giving a well-defined power series which can be taken as the definition of $r^q_m(x)$. Moreover, the rate at which the coefficients stabilize does not
    depend on $m$.
\end {thmB}

Another fundamental property of the $q$-rationals is positivity. It was shown in \cite{mgo_20} that if $\frac{r}{s} > \frac{a}{b}$, then $\mathcal{R}(q)\mathcal{B}(q) - \mathcal{A}(q)\mathcal{S}(q)$ has positive coefficients, where $\mathcal{R}$, $\mathcal{S}$, $\mathcal{A}$, $\mathcal{B}$ are the numerators and denominators of the corresponding $q$-rationals. Basser, Ovenhouse, and Sakarda have shown that in the $q=1$ case, the higher continued fractions $r_{i,m}(x)$ are increasing functions on $\mathbb{R}_{>0}$ \cite{BOS}. Our final main result
is that this positivity property also holds for the higher $q$-rationals.

\begin {thmC}
    Suppose $\frac{r}{s} > \frac{a}{b} \geq 1$. Let $0 \leq i \leq m$, and denote $r_m^q \left( \frac{r}{s} \right) = \frac{\mathcal{R}(q)}{\mathcal{S}(q)}$ 
    and $r_m^q \left( \frac{a}{b} \right) = \frac{\mathcal{A}(q)}{\mathcal{B}(q)}$.
    Then $\mathcal{R}(q)\mathcal{B}(q) - \mathcal{A}(q)\mathcal{S}(q)$ is a polynomial with positive integer coefficients.
\end {thmC}

The structure of the rest of the paper is as follows. In Section \ref{sec:enumeration}, we establish notation for our elementary matrices, and then we
re-state and prove (in the current language of $P$-partitions) a theorem from \cite{mosz_23} which is the $q=1$ version of Theorem A.
In Section \ref{sec:q_enumeration}, we define $q$-deformations of the elementary matrices, and we prove Theorem A. In Section \ref{sec:higher_q_cfs},
we review the definition and properties of Morier-Genoud and Ovsienko's $q$-rationals, and we define the $q$-analogs of the higher continued fractions.
In Section \ref{sec:stabilization}, we prove Theorem B, establishing the stabilization phenomenon for higher $q$-rationals.
In Section \ref{sec:positivity}, we prove Theorem C, establishing the positivity phenomenon.

\section {Enumeration of Plane Partitions} \label{sec:enumeration}

Let $R_m$ and $L_m$ be the upper and lower triangular $(m+1) \times (m+1)$ matrices with $(R_m)_{ij} = 1$ for $i \leq j$ and $(L_m)_{ij} = 1$ for $i \geq j$.
For example, when $m=1,2,3$, these look like:
\[ R_1 = \begin{pmatrix}1&1\\0&1\end{pmatrix}, \quad R_2 = \begin{pmatrix}1&1&1\\0&1&1\\0&0&1\end{pmatrix}, \quad R_3 = \begin{pmatrix}1&1&1&1\\0&1&1&1\\0&0&1&1\\0&0&0&1\end{pmatrix} \]
The number of $P$-partitions of $\mathcal{G}[a_1,\dots,a_n]$ with parts at most $m$ is related to the matrix product $R_m^{a_1}L_m^{a_2}R_m^{a_3}L_m^{a_4} \cdots$ \cite{mosz_23}. 
This matrix product can be expressed in another equivalent way, which we will now explain.
Let $W_m$ be the following $(m+1)\times(m+1)$ anti-diagonal matrix
\[ W_m := \begin{pmatrix}
0&0&\cdots&0&1\\
0&0&\cdots&1&0\\
&&\iddots\\
0&1&\cdots &0&0\\
1&0&\cdots &0&0
\end{pmatrix},
\] 
and define $\Lambda_m(a) = R_m^aW_m = W_mL_m^a$. It is
not hard to check that the entries of this matrix are $\Lambda_m(a)_{ij} = \mbinom{a}{m+2-i-j} = \binom{a+m+1-i-j}{m+2-i-j} = \binom{a+m+1-i-j}{a-1}$.\footnote{We are using the ``multichoose''
notation $\mbinom{n}{k} := \binom{n+k-1}{k}$.} 
In other words, $\Lambda_m(a)$ is a Hankel matrix, with every entry on the $k^\mathrm{th}$ skew diagonal equal to $\mbinom{a}{k} = \binom{a+k-1}{k}$.
Since $\Lambda_m(a)\Lambda_m(b) = R_m^a W_m^2L_m^b = R_m^a L_m^b$, we can also equivalently consider the matrix product $\Lambda_m(a_1)\Lambda_m(a_2) \cdots \Lambda_m(a_n)$. The main result of
\cite{mosz_23} gives a combinatorial interpretation of the entries of this matrix product. We will re-state a slightly different, but equivalent, version of this
theorem now, after a couple more definitions.

Let $\mathcal{G} = \mathcal{G}[a_1,\dots,a_n]$. Recall that $\Omega_m(\mathcal{G})$ is the set of $P$-partitions of $\mathcal{G}$ with parts at most $m$.
The boxes of $\mathcal{G}$ (and hence the parts of a $P$-partition) are naturally indexed from bottom-left to top-right. Let $\Omega_m^{i,j}(\mathcal{G})$ be the subset of $\Omega_m(\mathcal{G})$ consisting of
$P$-partitions whose first part is at most $m+1-i$ and whose last part is at most $m+1-j$. Similarly, let $\overline{\Omega}_m^{i,j}(\mathcal{G})$ be those
with first part at most $m+1-i$ and last part at least $j-1$.
In particular, $\Omega_m^{1,1}(\mathcal{G}) = \overline{\Omega}_m^{1,1}(\mathcal{G}) = \Omega_m(\mathcal{G})$.
The following is an equivalent restatement of Theorem 3.8 from \cite{mosz_23} (although phrased in somewhat different terminology).

\begin {theorem} [\cite{mosz_23}] \label{thm:matrix_formula}
    Let $(a_1,a_2,\dots,a_n)$ be a sequence of positive integers, and let $\mathcal{G} = \mathcal{G}[a_1,\dots,a_n]$. Define $X$ as the following matrix product:
    \[ X := \Lambda_m(a_1) \Lambda_m(a_2) \cdots \Lambda_m(a_n) \]
    Then $X_{ij} = \left| \Omega_m^{i,j}(\mathcal{G}) \right|$ if $n$ is even, and $X_{ij} = \left| \overline{\Omega}_m^{i,j}(\mathcal{G}) \right|$ if $n$ is odd. 
    In particular, we always have $X_{11} = \left| \Omega_m(\mathcal{G}) \right|$.
\end {theorem}
\begin {proof}
    Induct on $n$. If $n=1$, then $\mathcal{G}$ is a straight vertical column of $a_1-1$ boxes.
    It is well-known (and easy to check) that the number of $P$-partitions of a $p$-element chain with parts at most $k$ is equal to $\mbinom{p+1}{k} = \binom{p+k}{k}$.
    Clearly the number of $P$-partitions whose image is in the range $\{r,r+1,\dots,s\}$ is the same as if the image were $\{0,1,\dots,s-r\}$.
    The $P$-partitions we want to consider have image in the range 
    $\{j-1,\dots,m+1-i\}$, which is the same as those with image in the range $\{0,1,\dots,m+2-i-j\}$, giving $\mbinom{a_1}{m+2-i-j}$ in total.
    This is indeed the $(i,j)$-entry of the matrix $\Lambda_m(a_1)$.

    Now suppose the result is true for $X' = \Lambda_m(a_1) \cdots \Lambda_m(a_{n-1})$. We need to check that the matrix entries of $X = X' \Lambda_m(a_n)$
    give the correct counts. The entries of $X$ are given by
    \[ X_{ij} = \sum_{k=1}^{m+1} X'_{ik} \Lambda_m(a_n)_{kj} = \sum_{k=1}^{m+1} X'_{ik} \mbinom{a_n}{m+2-k-j} \]

    Suppose that $n$ is even (so $n-1$ is odd). 
    Then by induction, $X'_{ik} = \overline{\Omega}^{i,k}_m(\mathcal{G}')$. The case where $n$ is odd is similar, and will be omitted. The shape
    $\mathcal{G}$ is obtained from $\mathcal{G}'$ by adding one more box above the last box of $\mathcal{G}'$ 
    and then adding a horizontal segment of $a_n-1$ boxes going to the right, for a total of $a_n$ new boxes. 
    Consider this new ``corner box'' where the new horizontal segment joins to the old diagram, and suppose a $P$-partition assigns
    $k-1$ to this box. Then all boxes beneath and to the right must be at least $k-1$. Therefore the set of all such $P$-partitions is the product
    $\overline{\Omega}_m^{i,k}(\mathcal{G}') \times S$, where $S$ is the set of $P$-partitions of the horizontal segment with all parts in the range $\{k-1,\dots,m+1-j\}$.
    By the argument from the previous paragraph, $|S| = \mbinom{a_n}{m+2-k-j}$, and so we have 
    $\left| \overline{\Omega}_m^{i,k}(\mathcal{G}') \times S \right| = X'_{ik} \mbinom{a_n}{m+2-k-j}$.
    Summing over all possible $k$ gives the result.
\end {proof}
\begin {example} \label{ex:m=2_matrix_product}
    Consider $\mathcal{G}[2,2]$. For $m=2$, we have
    \[ X = R^2 L^2 = \Lambda(2)^2 = \begin{pmatrix} 14 & 8 & 3 \\ 8 & 5 & 2 \\ 3 & 2 & 1 \end{pmatrix} \]
    The upper-left entry tells us that there are $14$ $P$-partitions of shape $\mathcal{G}[2,2]$ with parts at most 2. These $14$ $P$-partitions are pictured in Figure \ref{fig:poset}.
    One can also see that the other matrix entries count the appropriate subsets. For example, the $(1,2)$-entry of the matrix is $8$,
    corresponding to the 8 $P$-partitions where the label in the top-right box is at most $1$.

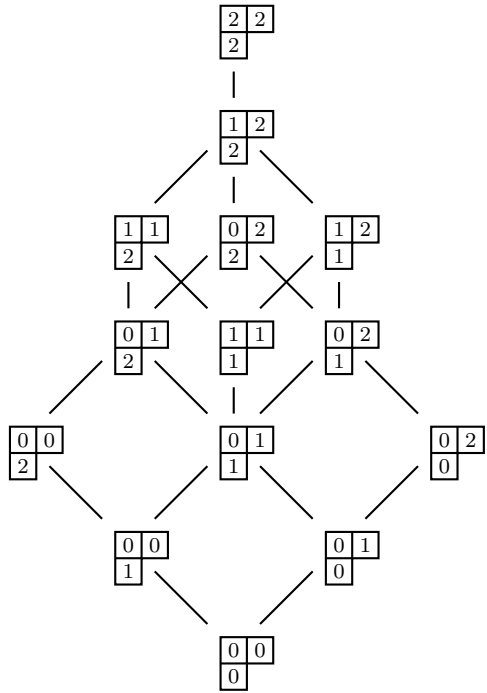
\begin {figure}[h]
\centering
\caption {The poset of $P$-partitions of $\mathcal{G}[2,2]$ with parts at most $2$.}
\label {fig:poset}
\begin {tikzpicture}[scale=0.35]
    \newcommand{\drawsnake}{
        \draw (0,0) -- (0,1) -- (3,1) -- (3,0) -- cycle;
        \draw (1,0) -- (1,1);
        \draw (2,0) -- (2,1);
    }

    \newcommand{\drawdualsnake}{
        \draw (0,0) -- (0,2) -- (2,2) -- (2,1) -- (1,1) -- (1,0) -- cycle;
        \draw (0,1) -- (1,1) -- (1,2);
    }

    \newcommand{\drawposet}{
        \begin {scope}[shift={(-0.5,0)}]
        \draw[fill=black] (0,0) circle (0.05);
        \draw[fill=black] (1,1) circle (0.05);
        \draw[fill=black] (2,0) circle (0.05);
        \draw[fill=black] (0,1) circle (0.05);
        \draw[fill=black] (1,2) circle (0.05);
        \draw[fill=black] (2,1) circle (0.05);

        \draw (0,0) -- (1,1) -- (2,0);
        \draw (0,1) -- (1,2) -- (2,1);
        \draw (0,0) -- (0,1);
        \draw (1,1) -- (1,2);
        \draw (2,0) -- (2,1);
        \end {scope}
    }

    %%% Rank 0 %%%

    \drawdualsnake

    \draw (0.5,0.5) node {\tiny $0$};
    \draw (0.5,1.5) node {\tiny $0$};
    \draw (1.5,1.5) node {\tiny $0$};

    \draw (-0.5,2.5) -- (-2.5,4.5);
    \draw (1.5,2.5) -- (3.5,4.5);

    %%% Rank 1 %%%

    \begin {scope}[shift={(-4,4)}]
        \drawdualsnake

        \draw (0.5,0.5) node {\tiny $1$};
        \draw (0.5,1.5) node {\tiny $0$};
        \draw (1.5,1.5) node {\tiny $0$};

        \draw (-0.5,2.5) -- (-2.5,4.5);
        \draw (1.5,2.5) -- (3.5,4.5);
    \end {scope}

    \begin {scope}[shift={(4,4)}]
        \drawdualsnake

        \draw (0.5,0.5) node {\tiny $0$};
        \draw (0.5,1.5) node {\tiny $0$};
        \draw (1.5,1.5) node {\tiny $1$};

        \draw (-0.5,2.5) -- (-2.5,4.5);
        \draw (1.5,2.5) -- (3.5,4.5);
    \end {scope}

    %%% Rank 2 %%%

    \begin {scope}[shift={(-8,8)}]
        \drawdualsnake

        \draw (0.5,0.5) node {\tiny $2$};
        \draw (0.5,1.5) node {\tiny $0$};
        \draw (1.5,1.5) node {\tiny $0$};

        \draw (1.5,2.5) -- (3.5,4.5);
    \end {scope}

    \begin {scope}[shift={(0,8)}]
        \drawdualsnake

        \draw (0.5,0.5) node {\tiny $1$};
        \draw (0.5,1.5) node {\tiny $0$};
        \draw (1.5,1.5) node {\tiny $1$};

        \draw (0.5,2.5) -- (0.5,3.5);
        \draw (-0.5,2.5) -- (-2.5,4.5);
        \draw (1.5,2.5) -- (3.5,4.5);
    \end {scope}

    \begin {scope}[shift={(8,8)}]
        \drawdualsnake

        \draw (0.5,0.5) node {\tiny $0$};
        \draw (0.5,1.5) node {\tiny $0$};
        \draw (1.5,1.5) node {\tiny $2$};

        \draw (-0.5,2.5) -- (-2.5,4.5);
    \end {scope}

    %%% Rank 3 %%%

    \begin {scope}[shift={(-4,12)}]
        \drawdualsnake

        \draw (0.5,0.5) node {\tiny $2$};
        \draw (0.5,1.5) node {\tiny $0$};
        \draw (1.5,1.5) node {\tiny $1$};

        \draw (0.5,2.5) -- (0.5,3.5);
        \draw (1.5,2.5) -- (3.5,4.5);
    \end {scope}

    \begin {scope}[shift={(0,12)}]
        \drawdualsnake

        \draw (0.5,0.5) node {\tiny $1$};
        \draw (0.5,1.5) node {\tiny $1$};
        \draw (1.5,1.5) node {\tiny $1$};

        \draw (-0.5,2.5) -- (-2.5,4.5);
        \draw (1.5,2.5) -- (3.5,4.5);
    \end {scope}

    \begin {scope}[shift={(4,12)}]
        \drawdualsnake

        \draw (0.5,0.5) node {\tiny $1$};
        \draw (0.5,1.5) node {\tiny $0$};
        \draw (1.5,1.5) node {\tiny $2$};

        \draw (-0.5,2.5) -- (-2.5,4.5);
        \draw (0.5,2.5) -- (0.5,3.5);
    \end {scope}

    %%% Rank 4 %%%

    \begin {scope}[shift={(-4,16)}]
        \drawdualsnake

        \draw (0.5,0.5) node {\tiny $2$};
        \draw (0.5,1.5) node {\tiny $1$};
        \draw (1.5,1.5) node {\tiny $1$};

        \draw (1.5,2.5) -- (3.5,4.5);
    \end {scope}

    \begin {scope}[shift={(0,16)}]
        \drawdualsnake

        \draw (0.5,0.5) node {\tiny $2$};
        \draw (0.5,1.5) node {\tiny $0$};
        \draw (1.5,1.5) node {\tiny $2$};

        \draw (0.5,2.5) -- (0.5,3.5);
    \end {scope}

    \begin {scope}[shift={(4,16)}]
        \drawdualsnake

        \draw (0.5,0.5) node {\tiny $1$};
        \draw (0.5,1.5) node {\tiny $1$};
        \draw (1.5,1.5) node {\tiny $2$};

        \draw (-0.5,2.5) -- (-2.5,4.5);
    \end {scope}

    %%% Rank 5 %%%

    \begin {scope}[shift={(0,20)}]
        \drawdualsnake

        \draw (0.5,0.5) node {\tiny $2$};
        \draw (0.5,1.5) node {\tiny $1$};
        \draw (1.5,1.5) node {\tiny $2$};

        \draw (0.5,2.5) -- (0.5,3.5);
    \end {scope}

    %%% Rank 6 %%%

    \begin {scope}[shift={(0,24)}]
        \drawdualsnake

        \draw (0.5,0.5) node {\tiny $2$};
        \draw (0.5,1.5) node {\tiny $2$};
        \draw (1.5,1.5) node {\tiny $2$};
    \end {scope}
\end {tikzpicture}
\end {figure}
\end {example}

\smallskip

\begin {remark}
    There is a simple bijection between $\Omega_m(\mathcal{G})$ and the set of \emph{$m$-lattice paths}, which are $m$-tuples of north-east lattice paths
    on $\mathcal{G}$ (up to equivalence, where two tuples of paths are identified if they have the same multiset of edges). The labels in the $P$-partition
    are the number of paths going above that box. 

    There is a bijection between north-east lattice paths on $\mathcal{G}$ and perfect matchings (also called \emph{dimer covers}) of a different snake graph
    $\mathcal{G}^*$ (called the \emph{dual snake graph}), as explained in \cite{propp} and \cite{claussen}. This bijection extends naturally to a bijection between
    $m$-lattice paths and $m$-dimer covers ($m$-tuples of perfect matchings, up to equivalence).

    The $m=1$ dimer covers of snake graphs were used in \cite{ms_10,msw_11} to compute cluster variables as well as their $F$-polynomials in cluster algebras of surface type. The $m=2$ double dimer covers of snake graphs were used in \cite{musiker2021expansion,musiker2022double} to compute cluster variables in certain ``supersymmetric'' cluster algebras.
    Through the above bijections, these results can be reformulated in terms of $P$-partitions of skew Young diagrams. 
\end {remark}

\begin{example}\label{ex kronecker}

We illustrate this correspondence in the example of cluster algebra of the Kronecker quiver  $\xymatrix{1\ar@<2pt>[r]\ar@<-2pt>[r] &2}$. The corresponding surface is an annulus with one marked point on each boundary component. 
In the associated cluster algebra, consider the cluster variable obtained by the sequence of mutations at 2,1,2. Its snake graph $\calG^*[2,1,1,2]$ (as defined in \cite{cs_18}) with face labels corresponding to principal coefficients is shown in the top left picture in Figure~\ref{fig Fpoly}.  The corresponding skew Young diagram $\calG[2,1,1,2]$ is shown in the top right picture of the same figure. The second and the third row illustrate the bijection between 1-dimer covers of $\calG^*[2,1,1,2]$ and $P$-partitions of $\calG[2,1,1,2]$  in two examples. The left picture in both rows shows a dimer cover of $\calG^*[2,1,1,2]$, %the picture in the middle shows the corresponding lattice path in $\calG[2,1,1,2]$ 
and the right picture shows the corresponding $P$-partition.
  
The dimer cover in the second row is the minimal perfect matching and its contribution to the $F$-polynomial is 1. 
The dimer cover in the third row is obtained from the minimal matching by twisting  the first, third, fourth and fifth tile. 
The corresponding tile labels are 2,2,2,1. The matching therefore contributes the monomial $y_1 y_2^3$ to the $F$-polynomial. The monomials can also be recovered from %the middle picture by considering the tiles that lie below the lattice path, and from 
the right picture by considering the labels in the tiles as the exponents of the monomial.

\end{example}

 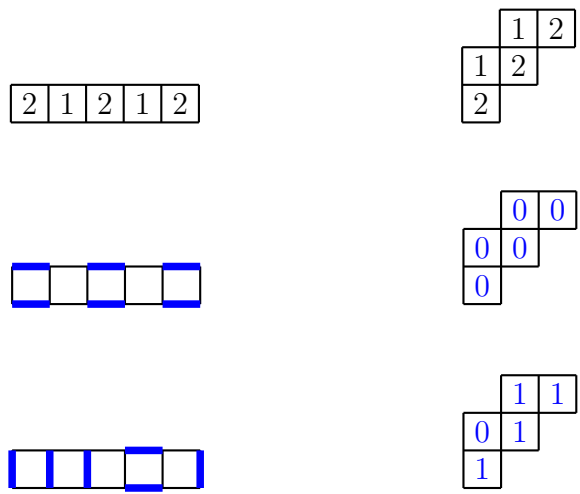
\begin{figure}
\begin{center}

\[\begin {tikzpicture}[scale=0.5]
    \draw (0,0) -- (5,0);
    \draw (0,1) -- (5,1);
    \draw (0,0) -- (0,1);
    \draw (1,0) -- (1,1);
    \draw (2,0) -- (2,1);
    \draw (3,0) -- (3,1);
    \draw (4,0) -- (4,1);
    \draw (5,0) -- (5,1);
    \node at (0.5,0.5) {2};
    \node at (1.5,0.5) {1};
    \node at (2.5,0.5) {2};
    \node at (3.5,0.5) {1};
    \node at (4.5,0.5) {2};
  
   \draw (12,0) -- (13,0);
   \draw (12,0) -- (12,2);
   \draw (13,0) -- (13,3);
   \draw (14,1) -- (14,3);
   \draw (15,2) -- (15,3);
   \draw (12,1) -- (14,1);
   \draw (12,2) -- (15,2);
   \draw (13,3) -- (15,3);
   
    \node at (12.5,0.5) {2};
    \node at (12.5,1.5) {1};
    \node at (13.5,1.5) {2};
    \node at (13.5,2.5) {1};
    \node at (14.5,2.5) {2};

     \end {tikzpicture}
    \]
    
    % second row
    \[\begin {tikzpicture}[scale=0.5]
    \draw (0,0) -- (5,0);
    \draw (0,1) -- (5,1);
    \draw (0,0) -- (0,1);
    \draw (1,0) -- (1,1);
    \draw (2,0) -- (2,1);
    \draw (3,0) -- (3,1);
    \draw (4,0) -- (4,1);
    \draw (5,0) -- (5,1);
%    \node at (0.5,0.5) {2};
%    \node at (1.5,0.5) {1};
%    \node at (2.5,0.5) {2};
%    \node at (3.5,0.5) {1};
%    \node at (4.5,0.5) {2};
    
     \draw [line width=1mm,  blue] (0,0) -- (1,0) ;
     \draw [line width=1mm,  blue] (0,1) -- (1,1) ;
     \draw [line width=1mm,  blue] (2,0) -- (3,0) ;
     \draw [line width=1mm,  blue] (4,0) -- (5,0) ;
     \draw [line width=1mm,  blue] (2,1) -- (3,1) ;
     \draw [line width=1mm,  blue] (4,1) -- (5,1) ;

%     
%     \draw [line width=1mm,  blue] (7,0) -- (8,0) ;
%     \draw [line width=1mm,  blue] (8,0) -- (8,1) ;
%     \draw [line width=1mm,  blue] (8,1) -- (9,1) ;
%     \draw [line width=1mm,  blue] (9,2) -- (9,1) ;         \draw [line width=1mm,  blue] (9,2) -- (10,2) ;         \draw [line width=1mm,  blue] (10,3) -- (10,2) ;
%
%   \draw (7,0) -- (8,0);
%   \draw (7,0) -- (7,2);
%   \draw (8,0) -- (8,3);
%   \draw (9,1) -- (9,3);
%   \draw (10,2) -- (10,3);
%   \draw (7,1) -- (9,1);
%   \draw (7,2) -- (10,2);
%   \draw (8,3) -- (10,3);

   \draw (12,0) -- (13,0);
   \draw (12,0) -- (12,2);
   \draw (13,0) -- (13,3);
   \draw (14,1) -- (14,3);
   \draw (15,2) -- (15,3);
   \draw (12,1) -- (14,1);
   \draw (12,2) -- (15,2);
   \draw (13,3) -- (15,3);

   \node[blue] at (12.5,0.5) {0}; 
\node[blue]  at (12.5,1.5) {0}; 
\node[blue]  at (13.5,1.5) {0}; 
\node[blue]  at (13.5,2.5) {0}; 
\node[blue]  at (14.5,2.5) {0};

     \end {tikzpicture}
\]
% third row

    \[\begin {tikzpicture}[scale=0.5]
    \draw (0,0) -- (5,0);
    \draw (0,1) -- (5,1);
    \draw (0,0) -- (0,1);
    \draw (1,0) -- (1,1);
    \draw (2,0) -- (2,1);
    \draw (3,0) -- (3,1);
    \draw (4,0) -- (4,1);
    \draw (5,0) -- (5,1);
%    \node at (0.5,0.5) {2};
%    \node at (1.5,0.5) {1};
%    \node at (2.5,0.5) {2};
%    \node at (3.5,0.5) {1};
%    \node at (4.5,0.5) {2};
    
     \draw [line width=1mm,  blue] (0,0) -- (0,1) ;
     \draw [line width=1mm,  blue] (1,0) -- (1,1) ;
     \draw [line width=1mm,  blue] (2,0) -- (2,1) ;
     \draw [line width=1mm,  blue] (4,0) -- (3,0) ;
     \draw [line width=1mm,  blue] (4,1) -- (3,1) ;
     \draw [line width=1mm,  blue] (5,0) -- (5,1) ;

%     
%     \draw [line width=1mm,  blue] (7,0) -- (7,1) ;
%     \draw [line width=1mm,  blue] (7,1) -- (8,1) ;
%     \draw [line width=1mm,  blue] (8,1) -- (8,3) ;
%     \draw [line width=1mm,  blue] (8,3) -- (10,3) ;         

%   \draw (7,0) -- (8,0);
%   \draw (7,0) -- (7,2);
%   \draw (8,0) -- (8,3);
%   \draw (9,1) -- (9,3);
%   \draw (10,2) -- (10,3);
%   \draw (7,1) -- (9,1);
%   \draw (7,2) -- (10,2);
%   \draw (8,3) -- (10,3);

   \draw (12,0) -- (13,0);
   \draw (12,0) -- (12,2);
   \draw (13,0) -- (13,3);
   \draw (14,1) -- (14,3);
   \draw (15,2) -- (15,3);
   \draw (12,1) -- (14,1);
   \draw (12,2) -- (15,2);
   \draw (13,3) -- (15,3);

   \node[blue]  at (12.5,0.5) {1}; 
\node[blue]  at (12.5,1.5) {0}; 
\node[blue]  at (13.5,1.5) {1}; 
\node[blue]  at (13.5,2.5) {1}; 
\node[blue]  at (14.5,2.5) {1};

     \end {tikzpicture}
\]

\caption{Illustration of the bijection in  Example \ref{ex kronecker}. %snake graph $\calG^*[2,1,1,2]$ and the corresponding skew Young diagram (dual snake graph) $\calG[2,1,1,2]$ 
}
\label{fig Fpoly}
\end{center}
\end{figure}

\bigskip

\section {Refined Enumeration and $q$-Analogs} \label{sec:q_enumeration}

One can see in Figure \ref{fig:poset} that there is a natural partial order to the set of $P$-partitions. The covering relations
are given by incrementing a single entry by 1. So rather than just counting $P$-partitions,
one could also ``$q$-count'' them. That is, we could consider the rank generating function of this poset of $P$-partitions with parts at most $m$. 
In the example, this is $1 + 2q + 3q^2 + 3q^3 + 3q^4 + q^5 + q^6$.

These generating functions are also the principal specializations of certain chromatic quasi-symmetric functions \cite{shareshian2016chromatic}
associated to the fence posets.

The main goal of this section is to give a $q$-analog of Theorem \ref{thm:matrix_formula}. To do so, we will need to define $q$-analogs
of the $R$, $L$, and $\Lambda$ matrices, and find explicit expressions for products of them.

We will use the standard notations $[n]_q = 1+q+q^2+\cdots+q^{n-1}$, $[n]_q! = [n]_q [n-1]_q \cdots [1]_q$, and $\binom{n}{k}_q = \frac{[n]_q!}{[k]_q![n-k]_q!}$.
We will also use the multichoose notation $\mbinom{n}{k}_q := \binom{n+k-1}{k}_q$.

\bigskip

Let $Q_m = \mathrm{diag}(q^m,q^{m-1},\dots,q,1)$ be the diagonal matrix with powers of $q$ on the diagonal. Define the $q$-deformations
of the $R$ and $L$ matrices by
\[ R_m(q) := R_m Q_m, \quad \quad L_m(q) := L_m Q_m \]
For example, when $m=2$, we have
\[ R_2(q) = \begin{pmatrix} q^2 & q & 1 \\ 0 & q & 1 \\ 0 & 0 & 1 \end{pmatrix}, \quad \quad L_2(q) = \begin{pmatrix} q^2 & 0 & 0 \\ q^2 & q & 0 \\ q^2 & q & 1 \end{pmatrix} \]

It will be useful to have formulas for powers of the $R_m(q)$ and $L_m(q)$ matrices. But first some basic remarks. 
A well-known identity of binomial coefficients (sometimes called \emph{Fermat's identity} or the \emph{hockey stick identity})
says that $\sum_{i=k}^n \binom{i}{k} = \binom{n+1}{k+1}$, which in multichoose notation becomes $\sum_{\ell=0}^k \mbinom{n}{\ell} = \mbinom{n+1}{k}$. 
There are the following $q$-analogs, which are easily derived from the two $q$-versions of Pascal's identity.

\begin {lemma} \label{lem:fermat}
    For integers $n$ and $k$,
    \begin {enumerate}
        \item[(a)] $\displaystyle \mbinom{n}{k}_q = \sum_{\ell=0}^k q^\ell \mbinom{n-1}{\ell}_q$
        \item[(b)] $\displaystyle \mbinom{n}{k}_q = \sum_{\ell=0}^k q^{(k-\ell)(n-1)} \mbinom{n-1}{\ell}_q$
    \end {enumerate}
\end {lemma}

The following is a $q$-analog of  \cite[Lemma 4.4]{mosz_23}. It
gives an expression for the matrix entries of powers of $R_m(q)$.

\begin {lemma} \label{lem:powers_of_R}
    Let $j=i+k$ with $k \geq 0$. The $(i,j)$-entry of $R_m(q)^a$ is given by
    \[ \left( R_m(q)^a \right)_{ij} = \left( R_m(q)^a \right)_{i,i+k} = q^{a(m+1-i-k)} \mbinom{a}{k}_q = q^{a(m+1-j)} \binom{a+j-i-1}{j-i}_q \]
\end {lemma}
\begin {proof}
    When $a=1$ it is clearly true, since $\mbinom{1}{k}_q = \binom{k}{k}_q = 1$, and $R_m(q)$ has all upper-triangular entries equal to $q^{m+1-j}$.
    In general, writing $R_m(q)^a = R_m(q) \cdot R_m(q)^{a-1}$ and using induction, we get
    \begin {align*}
        (R_m(q)^a)_{ij} &= (R_m(q) \cdot R_m(q)^{a-1})_{ij} \\
                     &= \sum_\ell R_m(q)_{i\ell} (R_m(q)^{a-1})_{\ell j} \\
                     &= \sum_{\ell = i}^j q^{m+1-\ell} q^{(a-1)(m+1-j)} \mbinom{a-1}{j-\ell}_q \\
                     &= \sum_{\ell = i}^j q^{a(m+1-j)} q^{j-\ell} \mbinom{a-1}{j-\ell}_q \\
                     &= q^{a(m+1-j)} \sum_{\ell = 0}^{j-i}  q^{\ell} \mbinom{a-1}{\ell}_q \\
    \end {align*}
    By Lemma \ref{lem:fermat}(a), the sum is equal to $\mbinom{a}{j-i}_q = \mbinom{a}{k}_q$.
\end {proof}

\bigskip

\begin {remark} \label{rmk:powers_of_L}
    Note that by definition $R_m(q) = R_m Q_m$ and $L_m(q) = L_m Q_m$. Since $L_m = R_m^\top$, we have $L_m(q) = Q_m^{-1} R_m(q)^\top Q_m$, 
    and hence $L_m(q)^a = Q_m^{-1} (R_m(q)^a)^\top Q_m$.
    One can use this relation to get the corresponding formula for the entries of $L_m(q)^a$:
    \[ (L_m(q)^a)_{ij} = q^{i-j} (R_m(q)^a)_{ji} \]
\end {remark}

\bigskip

\begin {example}
    When $m=1$, these are the quantized generators of $\mathrm{PSL}_2(\Bbb{Z})$ used in \cite{mgo_20}, and we have
    \[ R_1(q)^a = \begin{pmatrix} q^a & [a]_q \\ 0 & 1 \end{pmatrix}, \quad L_1(q)^a = \begin{pmatrix} q^a & 0 \\ q[a]_q & 1 \end{pmatrix} \]
    For $m=2$, we have
    \[
        R_2(q)^a = \begin{pmatrix}
            q^{2a} & q^a[a]_q & \mbinom{a}{2}_q \\
            0      & q^a      & [a]_q \\
            0      & 0        & 1 
        \end{pmatrix}, \quad
        L_2(q)^a = \begin{pmatrix}
            q^{2a}             & 0      & 0 \\
            q^{a+1}[a]_q       & q^a    & 0 \\
            q^2\mbinom{a}{2}_q & q[a]_q & 1 
        \end{pmatrix}
    \]
    For $m=3$, we have
    \[
        R_3(q)^a = \begin{pmatrix}
            q^{3a} & q^{2a}[a]_q & q^a\mbinom{a}{2}_q & \mbinom{a}{3}_q \\
            0      & q^{2a}      & q^a[a]_q           & \mbinom{a}{2}_q \\
            0      & 0           & q^a                & [a]_q \\
            0      & 0           & 0                  & 1
        \end{pmatrix}, \quad
        L_3(q)^a = \begin{pmatrix}
            q^{3a}                 & 0                  & 0      & 0 \\
            q^{2a+1}[a]_q          & q^{2a}             & 0      & 0 \\
            q^{a+2}\mbinom{a}{2}_q & q^{a+1}[a]_q       & q^a    & 0 \\
            q^3\mbinom{a}{3}_q     & q^2\mbinom{a}{2}_q & q[a]_q & 1
        \end{pmatrix}
    \]
\end {example}

\bigskip

Now we define the $q$-analogs of the $\Lambda$-matrices. 
Recall that $W_m$ is the $(m+1)\times(m+1)$ anti-diagonal matrix with $1$'s on the anti-diagonal.

\bigskip

\begin {definition}
    For a positive integer $a$, let $\Lambda^+_m(q,a) := R_m(q)^a W_m$, and $\Lambda^-_m(q,a) := W_m L_m(q)^a$
\end {definition}

Thus $\Lambda^+_m(q,a)$ is obtained from $R_m(q)^a$ by reversing the order of the columns and 
$\Lambda^-_m(q,a)$ is obtained from $L_m(q)^a$ by reversing the order of the rows.
\bigskip

\begin {example}
    We obtain $\Lambda^+_3(q,a)$ and $\Lambda^-_3(q,a)$ from the example above by permuting the rows and columns:
    \[
        \Lambda^+_3(q,a) = \begin{pmatrix}
            \mbinom{a}{3}_q & q^a\mbinom{a}{2}_q & q^{2a}[a]_q & q^{3a} \\
            \mbinom{a}{2}_q & q^a[a]_q           & q^{2a}      & 0      \\
            [a]_q           & q^a                & 0           & 0      \\
            1               & 0                  & 0           & 0      
        \end{pmatrix}, \quad
        \Lambda^-_3(q,a) = \begin{pmatrix}
            q^3\mbinom{a}{3}_q     & q^2\mbinom{a}{2}_q & q[a]_q & 1 \\
            q^{a+2}\mbinom{a}{2}_q & q^{a+1}[a]_q       & q^a    & 0 \\
            q^{2a+1}[a]_q          & q^{2a}             & 0      & 0 \\
            q^{3a}                 & 0                  & 0      & 0 
        \end{pmatrix}
    \]
\end {example}

\bigskip

\begin {remark} \label{rmk:lambda_entries}
    By replacing $j \mapsto m+2-j$ in the expressions from Lemma \ref{lem:powers_of_R}, one obtains formulas for the entries of $\Lambda^+_m(q,a)$:
    \[ \left( \Lambda^+_m(q,a) \right)_{ij} = q^{(j-1)a} \mbinom{a}{m+2-i-j}_q \]
    Similarly, replacing $i \mapsto m+2-i$ in the expression from Remark \ref{rmk:powers_of_L} gives formulas for $\Lambda^-_m(q,a)$:
    \[ \left( \Lambda^-_m(q,a) \right)_{ij} = q^{m+2-i-j} q^{(i-1)a} \mbinom{a}{m+2-i-j}_q \]
\end {remark}

\bigskip

\begin {remark}
    Since $W_m^2 = \mathrm{Id}$, we have $R_m(q)^a L_m(q)^b = \Lambda^+(q,a)\Lambda^-(q,b)$. It follows that for any even-length sequence $a_1,a_2,\dots,a_{2n}$, we have
    \[ R_m(q)^{a_1}L_m(q)^{a_2} \cdots R_m(q)^{a_{2n-1}} L_m(q)^{a_{2n}} = \Lambda^+(q,a_1)\Lambda^-(q,a_2) \cdots \Lambda^+(q,a_{2n-1}) \Lambda^-(q,a_{2n}) \]
    and for an odd-length sequence $a_1,\dots,a_{2n+1}$, we have
    \[ R_m(q)^{a_1}L_m(q)^{a_2} \cdots L_m(q)^{a_{2n}} R_m(q)^{a_{2n+1}} W_m = \Lambda^+(q,a_1)\Lambda^-(q,a_2) \cdots \Lambda^+(q,a_{2n+1}) \]
\end {remark}

\bigskip

\begin {remark}
    We can also express the product using just $\Lambda^+$ (and not $\Lambda^-$). Since $\binom{n}{k}_q = q^{k(n-k)}\binom{n}{k}_{q^{-1}}$, one
    can show that $\Lambda^-_m(q,a) = q^{am} \Lambda^+_m(q^{-1},a)$. Therefore the even-length $RL$-product is also equal to
    \[ q^{m(a_2+a_4+\cdots+a_{2n})} \Lambda^+_m(q,a_1) \Lambda^+_m(q^{-1},a_2) \cdots \Lambda^+_m(q,a_{2n-1}) \Lambda^+_m(q^{-1},a_{2n}) \]
    In the case $m=1$, this is the form that was used in \cite{mgo_20}.
\end {remark}

\begin {definition}
    For a border strip $\mathcal{G} = \mathcal{G}[a_1,a_2,\dots,a_n]$, let $X_{\mathcal{G}}(q)$ be the following matrix product:
    \[ X_{\mathcal{G}}(q) = \Lambda_m^+(q,a_1)\Lambda_m^-(q,a_2) \Lambda_m^+(q,a_3) \Lambda_m^-(q,a_4) \cdots \Lambda_m^\pm(q,a_n), \]
    where the last $\pm$ depends on the parity of $n$.
\end {definition}

We are now almost ready to give the main theorem of this section, which is the $q$-analog of Theorem \ref{thm:matrix_formula}. 
To state the theorem, we must define the proper generating functions for the sets $\Omega_m^{ij}(\mathcal{G})$ and $\overline{\Omega}_m^{ij}(\mathcal{G})$.
If $\sigma \in \Omega_m(\mathcal{G})$ is a reverse plane partition (thought of as a map $\sigma \colon \mathcal{G} \to \{0,1,2,\dots,m\}$\footnote{The domain 
of this map is really the set of square faces of the graph $\mathcal{G}$.}), then its \emph{weight}
is simply the sum of its values. That is, $\mathrm{wt}(\sigma) = \sum_{x \in \mathcal{G}} \sigma(x)$. We then define $\Omega_m^{ij}(\mathcal{G},q)$ to be the generating
function for the subset $\Omega_m^{ij}(\mathcal{G})$:
\[ \Omega_m^{ij}(\mathcal{G},q) = \sum_{\sigma \in \Omega_m^{ij}(\mathcal{G})} q^{\mathrm{wt}(\sigma)} \]
and similarly for $\overline{\Omega}_m^{ij}(\mathcal{G},q)$.

\begin {theorem} \label{thm:q_matrix_formula}
    Let $\mathcal{G} = \mathcal{G}[a_1,\dots,a_n]$ be a border strip, and $X_\mathcal{G}(q)$ be the corresponding matrix product. Then
    \[
        X_\mathcal{G}(q)_{ij} = \begin{cases}
            q^{m+1-j} \, \Omega_m^{ij}(\mathcal{G},q) & \text{ if $n$ is even} \\[1ex]
            q^{j-1} \, \overline{\Omega}_m^{ij}(\mathcal{G},q) & \text { if $n$ is odd}
        \end{cases}
    \]
\end {theorem}
\begin {proof}
    Consider a poset $P$ which is a chain of $p$ elements. It is well-known, and easy to see, that the generating function
    is $\Omega_k(P,q) = \mbinom{p+1}{k}_q$. One way to see this is to use the bijection between plane partitions and lattice paths in a $p \times k$ rectangle,
    and the generating function for lattice paths is given by $q$-binomial coefficients.

    Now consider again the case of a $p$-element chain, but now restrict the entries to be in the range $\{r, r+1, \dots, s\}$. By shifting all entries by $r$,
    there is a clear bijection with $P$-partitions whose parts are in the range $\{0,1,\dots,s-r\}$. However, the weights of these $P$-partitions will differ
    by $pr$. Therefore the generating function will be $q^{pr}\mbinom{p+1}{s-r}_q$.

    For the base case of the theorem, when $n=1$, $\mathcal{G}[a_1]$ is a vertical column of $a_1-1$ boxes. By definition, $\overline{\Omega}_m^{ij}(\mathcal{G})$ is the set of $P$ partitions
    on the $(a_1-1)$-element chain with parts in the range $\{j-1,j,\dots,m+1-i\}$. So using the argument above, and substituting $p=a_1-1$, $r=j-1$, and $s=m+1-i$, we get
    \[ \overline{\Omega}_m^{ij}(\mathcal{G},q) = q^{(a_1-1)(j-1)}\mbinom{a_1}{m+2-i-j}_q \]
    Comparing this with Remark \ref{rmk:lambda_entries}, we see that it differs from $\Lambda_m^+(q,a_1)_{ij}$ by a factor of $q^{j-1}$. This establishes the theorem in the case $n=1$.

    Now we induct on $n$. For now, consider the case when $n$ is even (and $n-1$ is odd). Let $\mathcal{G}' = \mathcal{G}[a_1,\dots,a_{n-1}]$. By definition, we have
    $X_\mathcal{G}(q) = X_{\mathcal{G}'}(q) \Lambda_m^-(q,a_n)$, and so entry-wise we have
    \begin {align*} 
        X_\mathcal{G}(q)_{ij} &= \sum_k X_{\mathcal{G}'}(q)_{ik} \Lambda_m^-(q,a_n)_{kj} \\
                    &= \sum_k q^{k-1} \overline{\Omega}_m^{ik}(\mathcal{G}',q) \Lambda_m^-(q,a_n)_{kj} \\
                    &= \sum_k q^{k-1} \overline{\Omega}_m^{ik}(\mathcal{G}',q) q^{m+2-k-j} q^{(k-1)a_n} \mbinom{a_n}{m+2-k-j}_q \\
                    &= \sum_k q^{m+1-j} q^{(k-1)a_n} \overline{\Omega}_m^{ik}(\mathcal{G}',q) \mbinom{a_n}{m+2-k-j}_q
    \end {align*}
    Next, let us compute $\Omega_m^{ij}(\mathcal{G},q)$ with a combinatorial argument, and compare with this formula. Since we are assuming $n$ is even, the graph ends
    with a horizontal segment of boxes going left-to-right. Consider the corner box joining $\mathcal{G}'$ with this new segment. Let us enumerate the $P$-partitions
    which assign a label of $(k-1)$ to this corner box. By the definition of $P$-partitions, all boxes below and to the right must be at least $k-1$. Choosing the labels
    below the corner box gives all terms in the generating function $\overline{\Omega}_m^{ik}(\mathcal{G}',q)$. Choosing all possible labelings of the boxes to the right
    of the corner box gives all $P$-partitions of an $(a_n-1)$-element chain with values in the range $\{k-1,k,\dots,m+j-1\}$.
    As discussed above, summing over these possibilities gives $q^{(a_n-1)(k-1)}\mbinom{a_n}{m+2-k-j}_q$. Since the labels in the boxes below and to the right of the corner
    can be chosen independently, we simply multiply these polynomials. Lastly, we also need to multiply by $q^{k-1}$ to account for the corner box. Now sum over all possible $k$ to get
    \[ \Omega_m^{ij}(\mathcal{G},q) = \sum_k \overline{\Omega}_m^{ik}(\mathcal{G}',q) q^{(k-1)a_n}\mbinom{a_n}{m+2-k-j}_q \]
    Comparing with the expression for $X_\mathcal{G}(q)_{ij}$ above, we see it differs only by a factor of $q^{m+1-j}$. This proves the theorem in the case when $n$ is even.
    
    The case when $n$ is odd is similar. The recurrence for the matrix entries is
    \[ X_{\mathcal{G}}(q)_{ij} = \sum_k q^{m+1-k}q^{(j-1)a_n} \Omega^{i,k}_m(\mathcal{G}',q) \mbinom{a_n}{m+2-k-j}_q \]
    and the recurrence for $\overline{\Omega}^{ij}_m(\mathcal{G},q)$ from a slightly modified version of the combinatorial argument above is
    \[ \overline{\Omega}_m^{ij}(\mathcal{G},q) = \sum_k q^{m+1-k}q^{(j-1)(a_n-1)}{\Omega}_m^{ik}(\mathcal{G}',q) \mbinom{a_n}{m+2-k-j}_q\]
    We see that they differ only by a factor of $q^{j-1}$, which completes the proof.
\end {proof}

\begin {example}
    For $\mathcal{G} = \mathcal{G}[2,2]$, pictured in Figure \ref{fig:poset}, we have
    \[
        X_{\mathcal{G}}(q) = \Lambda_2^+(2,q) \Lambda_2^-(2,q) = \begin{pmatrix}
            q^2 \Omega_2^{11}(\mathcal{G},q) & q \Omega_2^{12}(\mathcal{G},q) & \Omega_2^{13}(\mathcal{G},q) \\[0.8ex]
            q^2 \Omega_2^{21}(\mathcal{G},q) & q \Omega_2^{22}(\mathcal{G},q) & \Omega_2^{23}(\mathcal{G},q) \\[0.8ex]
            q^2 \Omega_2^{31}(\mathcal{G},q) & q \Omega_2^{32}(\mathcal{G},q) & \Omega_2^{33}(\mathcal{G},q) 
        \end{pmatrix}
    \]
    For example, $\Omega_2^{11}(\mathcal{G},q) = 1+2q+3q^2+3q^3+3q^4+q^5+q^6$, which is the full rank generating function for the poset $\Omega_2(\mathcal{G})$.
    As another example, $\Omega_2^{21}(\mathcal{G},q) = 1+2q+2q^2+2q^3+q^4$, which is the rank generating function for the sub-poset $\Omega_2^{21}(\mathcal{G})$,
    consisting of those plane partitions where the bottom-left box is labelled by either $0$ or $1$ (but not $2$).
\end {example}

\begin{remark}
The results in this section can be readily adapted to the equivalent setting of $m$-dimer covers on snake graphs.  Fix a continued fraction $[a_1,\dots,a_n]$.  In this setting, each $P$-partition in $\Omega_m(\calG[a_1,\dots,a_n])$ corresponds to an $m$-dimer cover on a certain planar graph (the dual snake graph $\calG^*[a_1,\dots,a_n]$).  As shown in \cite{mosz_23}, the $m$-dimer covers on $\calG^*[a_1,\dots,a_n]$ can be enumerated via a weighted sum of $m$-dimer covers on the graph $\calG[n]$ consisting of a straight column of $n-1$ squares (note that this graph is referred to as $\calG[1^n]$ in \cite{mosz_23}).  This extends to the higher $q$-rational case by modifying the weighted sum by appropriate powers of $q$ as follows.  

Let $\delta_{\max}$ be the maximal element of the poset $\calG[n]$, that is, the unique $m$-dimer cover of $\calG[n]$ with weight $m$ on the top edge and weight $0$ on all interior edges.  Then for any $\delta \in \Omega_m[n]$, we define $\wt(\delta) = |\delta \cap \delta_{\max}|$.  Let $\widetilde \Omega_m^{ij}[n]$ be the set of $m$-dimer covers of $\calG[n]$ with weight $i$ on the bottom edge and $j$ on the top edge.  We then define
$$^i\llbracket a_1,\dots,a_n\rrbracket^j_{m,q} := \sum_{\delta \in \widetilde \Omega_m^{ij}[n]} q^{\wt(\delta)} \prod_{i=1}^n  \binom{a_i + \delta(e_i) - 1}{\delta(e_i)}_q\,.$$
Following the methods of \cite{mosz_23}, it can be shown that the rank generating function for $\widetilde \Omega_m^{ij}[n]$ is given by $^i\llbracket a_1,\dots,a_n\rrbracket^j_{m,q}$.  Moreover, a straightforward adaptation of the proof of \cite[Lemma 3.9]{mosz_23} yields the same recurrence for $m$-dimer covers as in the proof of Theorem~\ref{thm:q_matrix_formula}.  Thus the entries of $X_{\calG[a_1,\dots,a_n]}(q)$ can be interpreted as (shifted) rank generating functions for the sets $\widetilde \Omega_m^{ij}[n]$.
\end{remark}

\bigskip

\section {Higher $q$-continued fractions} \label{sec:higher_q_cfs}

In \cite{mgo_20}, Morier-Genoud and Ovsienko defined a $q$-analog of rational numbers generalizing the classical $q$-integers $[n]_q = 1+q+q^2+\cdots+q^{n-1}$.
Their definition was a $q$-deformation of the continued fraction representation of a number. Specifically, if $x \in \Bbb{Q}$ has continued fraction
$x = [a_1,a_2,\dots,a_{2m}]$ with with all $a_i$'s positive integers, then their definition is
\[ 
    [x]_q := [a_1]_q + \cfrac{q^{a_1}}{
        [a_2]_{q^{-1}} + \cfrac{q^{-a_2}}{
            [a_3]_q + \cfrac{q^{a_3}}{
                [a_4]_{q^{-1}} + \cfrac{q^{-a_4}}{
                    \ddots + \cfrac{q^{a_{2m-1}}}{
                        [a_{2m}]_{q^{-1}}
                    }
                }
            }
        }
    }
\]

Since their introduction, several combinatorial interpretations of these polynomials have been given.
We now take a moment to collect some of these various known results. Let $\frac{r}{s} = [a_1,\dots,a_n]$ be a rational number $(\geq 1)$,
and let $\mathcal{G} = \mathcal{G}[a_1,\dots,a_n]$ be the corresponding snake graph,
and let $\mathcal{G}^*$ be the \emph{dual snake graph}\footnote{See \cite{claussen} for an explanation of the duality construction for snake graphs.}.
Let $\left[\frac{r}{s}\right]_q = \frac{\mathcal{R}(q)}{\mathcal{S}(q)}$ be Morier-Genoud and Ovsienko's $q$-rational. Then we have the following:
\begin {enumerate}
    \item[(a)] $\mathcal{R}(q) = \sum_I q^{|I|}$, where the sum is over order ideals of $\mathcal{G}$, thought of as a fence poset.

    \medskip

    \item[(b)] $\mathcal{R}(q) = \sum_{p} q^{|p|}$, where the sum is over all north-east lattice paths
               on $\mathcal{G}$, and $|p|$ is the area underneath the path. 

    \medskip

    \item[(c)] $\mathcal{R}(q) = \sum_m q^{\mathrm{ht}(m)}$, where the sum is over perfect matchings of the dual snake graph $\mathcal{G}^*$,
               and $\mathrm{ht}(m)$ is the height of the perfect matching.

    \medskip

    \item[(d)] Viewing $\mathcal{G}$ as a skew shape $\lambda / \mu$ (i.e. the difference of two Young diagrams), $q^{|\mu|}\mathcal{R}(q)$ is the number of $k \times N$ matrices over the finite field $\Bbb{F}_q$ in reduced row echelon form
               which are representatives of the Schubert cells $\mu \leq \nu \leq \lambda$, where $k = \sum_i a_{2i}$ and $N = \sum_i a_i$.

    \medskip

    \item[(e)] $q\mathcal{R}(q) + (1-q)\mathcal{S}(q)$ is the normalized Jones polynomial of the two-bridge link obtained as the closure of
               the rational tangle of $\frac{r}{s} = [a_1,\dots,a_n]$.

    \medskip

    \item[(f)] $\mathcal{R}(q) = F(q,q,\dots,q)$ is the specialization of the $F$-polynomial of a certain cluster variable in a cluster algebra
               whose initial seed is a type $A$ Dynkin quiver.
\end {enumerate}

Part $(a)$ appeared as Theorem 4 in \cite{mgo_20}. The equivalence between $(a)$ and $(b)$ is straightforward. The fact that $(b)$ and $(c)$ are equivalent
follows from a bijection that was observed in \cite{propp} and \cite{claussen}. The equivalence between $(a)$ and $(c)$ was observed in \cite{msw_13} (Theorem 5.4).
Part $(d)$ appeared in \cite{ovenhouse_23}. Parts $(e)$ and $(f)$ appeared in the appendices of \cite{mgo_20}.

\bigskip

In \cite{mosz_23}, the authors defined a notion of \emph{higher continued fractions}. After reviewing the construction, we will explain
how the ratios of the polynomials $\Omega_m(\mathcal{G},q)$ from the previous section are $q$-analogs of higher continued fractions.
The higher continued fractions are given by a series of maps $r_{i,m} \colon \Bbb{Q}_{\geq 1} \to \Bbb{Q}_{\geq 1}$
which are defined recursively by a formula similar to the definition of ordinary continued fractions. Indeed, $r_{1,1}$ is the identity map, and
the corresponding recursion is simply the ordinary continued fraction definition. More generally, we have the following.

\begin {definition} \cite{mosz_23}
    Let $x \in \Bbb{Q}$ with $x \geq 1$, and let $0 \leq i \leq m$ be integers. Define $r_{0,m}(x) = 1$ for any $x$,
    and for an integer $n$, define $r_{i,m}(n) := \mbinom{n}{i} = \binom{n+i-1}{i}$.
    Otherwise if $x = [a_1,a_2,\dots,a_k]$ and $x' = [a_2,a_3,\dots,a_k]$, then
    \[ r_{i,m}(x) := \frac{1}{r_{m,m}(x')} \sum_{k=0}^i \mbinom{a_1}{k} r_{m-i+k,m}(x') \]
    We also define the vector $\mathrm{CF}_m(x) = \left( r_{m,m}(x), \, r_{m-1,m}(x), \dots r_{1,m}(x), \, r_{0,m}(x) \right)^\top$, which we
    call the \emph{$m$-continued fraction} of $x$.
\end {definition}

\bigskip

\begin {example}
    When $m=1$, we have $\mathrm{CF}_1(x) = (x,1)$ for any $x$, and the recurrence is the usual continued fraction recurrence:
    \[ [a_1,\dots,a_n] = a_1 + \frac{1}{[a_2,\dots,a_n]} \]
\end {example}

\begin {example}
    When $m=2$, we have $\mathrm{CF}_2(x) = \left( r_{2,2}(x), \, r_{1,2}(x), \, 1 \right)$, and these numbers satisfy the recurrences:
    \[ r_{1,2}(x) = a_1 + \frac{r_{1,2}(x')}{r_{2,2}(x')}, \quad \quad r_{2,2}(x) = \mbinom{a_1}{2} + a_1 \frac{r_{1,2}(x')}{r_{2,2}(x')} + \frac{1}{r_{2,2}(x')} \]
    For example, the continued fraction for $17/3$ is $[5,1,2]$, and using these recurrences, we get
    \[ 
        \mathrm{CF}_2\left( \frac{17}{3} \right) = \left( \frac{59}{3}, \, \frac{35}{6}, \, 1 \right) 
    \]
\end {example}

\begin {remark}
This recursive definition for $r_{im}(x)$ was originally designed to mimic the recursion of successively multiplying a sequence of $\Lambda(a_i)$ matrices.
Indeed, if $x = [a_1,\dots,a_n]$ and $x' = [a_2,\dots,a_n]$, with corresponding snake graphs $\mathcal{G}$ and $\mathcal{G}'$, then
\[ r_{im}(x) = \frac{\left| \Omega_m^{m+1-i,1}(\mathcal{G}) \right|}{\left| \Omega_m(\mathcal{G}') \right|} \]
(see \cite{mosz_23}, Theorem 6.4).
\end {remark}

In a similar fashion, we can define a $q$-analog of the $r_{i,m}$ recurrence which coincides with the multiplication of the $q$-deformed matrices.

\bigskip

\begin {definition} \label{def:r_im_q_recurrence}
    Define $r^q_{i,m} \colon \Bbb{Q}_{\geq 1} \to \Bbb{Z}(q)$ as follows. For $i=0$, let $r^q_{0,m}(x) = 1$ for any $x$. For $x=n \in \Bbb{N}$, we define
    \[ r^q_{i,m}(n) := \mbinom{n}{i}_q = \binom{n+i-1}{i}_q. \]
    In all other cases, if $x = [a_1,\dots,a_n]$ and $x'= [a_2,\dots,a_n]$, we have the recursive definition
    \[ r^q_{i,m}(x) := \sum_{k=0}^i q^{ka_1}\mbinom{a_1}{i-k}_q \frac{r^{q^{-1}}_{m-k,m}(x')}{r^{q^{-1}}_{mm}(x')} \]
\end {definition}

\bigskip

\begin {example}
    From Example \ref{ex:m=2_matrix_product} we can see that $r_{2,2} \left( \frac{5}{2} \right) = \frac{14}{3} = 4 + \frac{2}{3}$. 
    From Definition \ref{def:r_im_q_recurrence}, the $q$-analog is given by
    \[ r^q_{2,2} \left( \frac{5}{2} \right) = \frac{1+2q+3q^2+3q^3+3q^4+q^5+q^6}{1+q+q^2} = q(1+2q+q^3) + \frac{[2]_q}{[3]_q} \]
\end {example}

\bigskip

\begin{remark} \label{rmk:quotient_formula}
	Let $\calG = \calG[a_1,\cdots,a_n]$ and $X_\calG(q)$ the matrix product for $\calG$. 
        In the $q=1$ case, the higher continued fractions $[r_{m,m}(\calG):\cdots:r_{1,m}(\calG):1]$ 
        viewed as a projective point agree with the first column of the matrix $X_{\mathcal{G}}$. In other words, we have
	\[r_{i,m}(x)={X_{m+1-i,1}\over X_{m+1,1}} = \frac{\left| \Omega^{m+1-i,1}_m(\calG) \right|}{\left| \Omega_{m}({\calG'}) \right|}\] 
        where $\calG'=\calG [1,a_2-1,a_3,\cdots,a_n]$\footnote{Note that $\calG[a_2,\dots,a_n]$ and $\calG[1,a_2-1,\dots,a_n]$ are isomorphic as graphs,
        but their planar embedding is different, which makes a difference in the polynomial $\Omega_m(\calG',q)$.}. 
        In the $q$-deformed case, we have the analogous identity
        \[r_{i,m}^q(x)=\frac{\Omega^{m+1-i,1}_m(\calG,q)}{\Omega_{m}({\calG',q})}.\]
\end{remark}

\section {The Stabilization Phenomenon} \label{sec:stabilization}

In \cite{mgo_22}, Morier-Genoud and Ovsienko observed an intriguing property of their $q$-rational numbers, which they dubbed the \emph{stabilization phenomenon}.
The observation is that one can extend the definition to define $q$-analogs of irrational numbers, which will be formal Laurent series with integer coefficients.
More specifically, we have the following result.

\begin {theorem} \cite{mgo_22}
    Let $x$ be an irrational number, and $x_1,x_2,x_3,\dots$ a sequence of rationals converging to $x$. Then the sequence $[x_n]_q$ 
    has a well-defined limit as a formal Laurent series. In other words, if $[x_n]_q = \sum_k c_{n;k}q^k$, then for each $k$ the sequence $c_{n;k}$
    is eventually constant. 
\end {theorem}

\bigskip

In the notation of the theorem, if the sequence $c_{1;k}$, $c_{2;k}$, $c_{3;k}$, $\dots$ is eventually constant, with $c_{n;k} = c_k$ for $n$ large enough,
we define $[x]_q$ as the Laurent series $\sum_k c_kq^k$. Moreover, Morier-Genoud and Ovsienko gave a description of the rate at which the coefficients
stabilize. If the irrational number $x$ has infinite continued fraction $[a_1,a_2,\dots]$, and if $x_n = [a_1,\dots,a_n]$ are its convergents, then
$[x_{2n}]_q$ and $[x_{2n-1}]_q$ agree (as Laurent series) up to the $q^{a_1+a_2+\cdots+a_n-1}$ term.

\bigskip

\begin {example}
    Consider the real number $\sec(7) = \frac{1}{\cos(7)} \approx 1.3264$, with continued fraction $[1,3,15,1,3,3,\dots]$. 
    Truncating the continued fraction here, we can compute the series
    for $[\sec(7)]_q$ accurately to more than 20 terms, giving
    \[ [\sec(7)]_q = 1 + q^4 - q^6 - q^7 + q^8 + 2q^9 - 3q^{11} - 2q^{12} + 3q^{13} + 5q^{14} - q^{15} - 8q^{16} - 4q^{17} + 9 q^{18} + 11q^{19} + \cdots \]
\end {example}

It was also shown in \cite{mosz_23} that the definition of higher continued fractions extends to irrational values.
Specifically if $x_n$ are the rational convergents of an irrational number $x$,
then the limit $\lim_{n \to \infty} r_{im}(x_n)$ exists, and we define $r_{im}(x)$ to be the value of this limit when $x$ is irrational.
We will show in this section that the stabilization phenomenon holds for the higher $q$-rationals. That is, we can define $r^q_{im}(x)$
as some formal Laurent series when $x$ is irrational.

Before proving the main result, we will take a brief detour to explain a combinatorial interpretation of the $R$ and $L$ matrices in terms
of lattice paths (since it will be used in the proof). Let $G$ be an edge-weighted directed graph, and specify two subsets of vertices $\{v_1,\dots,v_k\}$ and $\{w_1,\dots,w_n\}$,
where the $v_i$'s are all sources, and the $w_j$'s are all sinks. We will call such a graph (together with weights and choices of sources/sinks) a \emph{network}.
We then form the path weight matrix $M_G$ whose entries $m_{ij}$ are the generating functions for paths starting at $v_i$ and ending at $w_j$. That is,
\[ m_{ij} := \sum_{p \colon v_i \to w_j} \mathrm{wt}(p) \]
Here, the weight of a path is the product of the edge weights. There are two important and well-known properties of these matrices. First, concatenation of
networks corresponds to matrix multiplication. That is, if $G_1$ and $G_2$ are such networks (such that the number of sinks of $G_1$ equals the number of sources of $G_2$),
and we form the network $G_1 \# G_2$ obtained by identifying the sources of $G_2$ with the sinks of $G_1$, then $M_{G_1 \# G_2} = M_{G_1} M_{G_2}$.

Second, there is the famous Lindstr\"{o}m-Gessel-Viennot formula \cite{gv_89}, which says the following. 
If $G$ is planar, and $I$ and $J$ are subsets of the source and sink vertices respectively,
with $|I|=|J|=k$, then the $k \times k$ minor of $M_G$ with row set $I$ and column set $J$ is a sum over families of non-crossing paths from $I$ to $J$. That is,
\[ \det \left( M[I,J] \right) = \sum_{(p_1,\dots,p_k)} \mathrm{wt}(p_1) \cdots \mathrm{wt}(p_k) \]
where the sum is over all tuples of non-intersecting paths which begin at the vertices of $I$ and end at the vertices of $J$.

Recall the matrices $R_m$ and $L_m$, which are upper-triangular and lower-triangular matrices filled with $1$'s. It is easy to find planar networks whose path-weight matrices
are $R_m$ and $L_m$. We describe them now (see Figure \ref{fig:path_weight_matrices} for an illustration). 
Define $G$ to be the network with $m+1$ horizontal strands, directed left-to-right. In the middle of each strand, place a vertex. Between each adjacent pair
of strands, connect the vertices with a downward-pointing vertical arrow. The resulting network has path-weight matrix given by $R_m$. The same construction, but with the vertical
edges oriented upwards, corresponds to $L_m$. Since $R_m(q)$ and $L_m(q)$ are obtained by right-multiplication with a diagonal matrix, we can simply label the horizontal edges
with $q^m$, $q^{m-1}$, $\dots$, $q$, $1$ from top-to-bottom to obtain networks with path-weight matrices given by $R_m(q)$ and $L_m(q)$. 

\begin {figure}
\centering
\begin {tikzpicture}
    \foreach \x in {0,1,2} {
    \foreach \y in {0,1,2,3} {
        \draw[fill=black] (\x,\y) circle (0.06);
    }
    }

    \foreach \x in {0,1} {
    \foreach \y in {0,1,2,3} {
        \draw[-latex] (\x,\y) -- (\x+0.94,\y);
    }
    }

    \foreach \y in {1,2,3} {
        \draw[-latex] (1,\y) -- (1,\y-0.94);
    }

    \draw (1.6,3) node[above] {$q^3$};
    \draw (1.6,2) node[above] {$q^2$};
    \draw (1.6,1) node[above] {$q$};

    \draw (-0.1,0) node[left] {$4$};
    \draw (-0.1,1) node[left] {$3$};
    \draw (-0.1,2) node[left] {$2$};
    \draw (-0.1,3) node[left] {$1$};

    \draw (2.1,0) node[right] {$4$};
    \draw (2.1,1) node[right] {$3$};
    \draw (2.1,2) node[right] {$2$};
    \draw (2.1,3) node[right] {$1$};

    \draw (1,-0.75) node {$R_3(q)$};

    \begin {scope}[shift={(5,0)}]
        \foreach \x in {0,1,2} {
        \foreach \y in {0,1,2,3} {
            \draw[fill=black] (\x,\y) circle (0.06);
        }
        }

        \foreach \x in {0,1} {
        \foreach \y in {0,1,2,3} {
            \draw[-latex] (\x,\y) -- (\x+0.94,\y);
        }
        }

        \foreach \y in {0,1,2} {
            \draw[-latex] (1,\y) -- (1,\y+0.94);
        }

        \draw (1.6,3) node[above] {$q^3$};
        \draw (1.6,2) node[above] {$q^2$};
        \draw (1.6,1) node[above] {$q$};

        \draw (-0.1,0) node[left] {$4$};
        \draw (-0.1,1) node[left] {$3$};
        \draw (-0.1,2) node[left] {$2$};
        \draw (-0.1,3) node[left] {$1$};

        \draw (2.1,0) node[right] {$4$};
        \draw (2.1,1) node[right] {$3$};
        \draw (2.1,2) node[right] {$2$};
        \draw (2.1,3) node[right] {$1$};

        \draw (1,-0.75) node {$L_3(q)$};
    \end {scope}
\end {tikzpicture}
\caption {The networks whose path-weight matrices are $R_3(q)$ and $L_3(q)$.}
\label {fig:path_weight_matrices}
\end {figure}
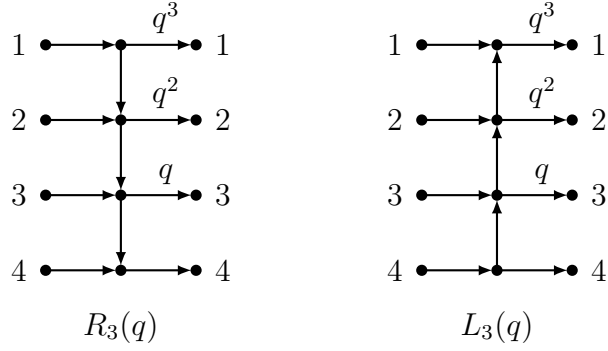

Recall that for $\mathcal{G} = \mathcal{G}[a_1,\dots,a_n]$, the matrix $X_\mathcal{G}(q)$ is given by
\begin {align*} 
    X_\mathcal{G}(q) &= \Lambda_m^+(q,a_1)\Lambda_m^-(q,a_2) \cdots \Lambda_m^+(q,a_{n-1})\Lambda_m^-(q,a_n) \\
           &= R_m(q)^{a_1}L_m(q)^{a_2} \cdots R_m(q)^{a_{n-1}}L_m(q)^{a_n} 
\end {align*}
if $n$ is even, and
\begin {align*} 
    X_\mathcal{G}(q) &= \Lambda_m^+(q,a_1)\Lambda_m^-(q,a_2) \cdots \Lambda_m^+(q,a_n) \\
           &= R_m(q)^{a_1}L_m(q)^{a_2} \cdots R_m(q)^{a_n}W_m
\end {align*}
if $n$ is odd. Since matrix multiplication corresponds to concatenation of networks, we can form a network corresponding to $X_\mathcal{G}(q)$ 
by attaching $a_1$ copies of the $R_m(q)$ network,
and then $a_2$ copies of the $L_m(q)$ network, and so on. If $n$ is even, this will correspond to $X_\mathcal{G}(q)$, and if $n$ is odd, it will correspond to $X_\mathcal{G}(q)W_m$.
If $x$ has continued fraction $[a_1,a_2,\dots,a_n]$, let us call this network $\Gamma_x$.

After this digression, we can now prove Theorem B from the introduction.

\begin {theorem} \label{thm:stabilization}
    Suppose $x$ is an irrational number with continued fraction $[a_1,a_2,a_3,\dots]$, and let $x_n = [a_1,\dots,a_n]$ be its
    rational convergents. The following statements hold for any $0 \leq i \leq m$. 
    \begin {enumerate}
        \item[(a)] If $n$ is even, then the Laurent series for $r^q_{im}(x_n)$ and $r^q_{im}(x_{n-1})$ agree up to $q^{a_1+a_2+\cdots+a_n-1}$.

        \item[(b)] If $n$ is odd, then the Laurent series for $r^q_{im}(x_n)$ and $r^q_{im}(x_{n-1})$ agree up to $q^{a_1+a_2+\cdots+a_{n-1}-1}$.
    \end {enumerate}
\end {theorem}

\begin{remark}
    The powers of $q$ in the theorem do not depend on $m$. In particular this means Morier-Genoud and Ovsienko's result (Proposition 1.1 from \cite{mgo_22} for the $m=1$ case)
    holds in general for all $m$.
\end{remark}

\begin {proof}
    We want to consider the difference 
    \begin {align*} 
        r^q_{im}(x_n) - r^q_{im}(x_{n-1}) &= \frac{\Omega_m^{m+1-i,1}(P,q)}{\Omega_m^{m+1,1}(P,q)} - \frac{\Omega_m^{m+1-i,1}(P',q)}{\Omega_m^{m+1,1}(P',q)} \\
                                          &= \frac{\Omega_m^{m+1-i,1}(P,q)\Omega_m^{m+1,1}(P',q) - \Omega_m^{m+1-i,1}(P',q)\Omega_m^{m+1,1}(P,q)}{\Omega_m^{m+1,1}(P,q)\Omega_m^{m+1,1}(P',q)}
    \end {align*}
    To prove the theorem, we want to show that as a Laurent series, this difference has lowest term $q^N$ (where $N$ depends on the parity of $n$).
    Note that the denominator is a polynomial with constant term 1, and so its reciprocal is a power series with constant term $1$. It thus suffices to show that the
    numerator has minimal term $q^N$ for the appropriate $N$.

    The specific details of the proof will differ in the even and odd case, but the main idea in both cases is that the numerator of this difference coincides
    (up to a factor of $q^k$ for some $k$) with the $2 \times 2$ minor of the matrix $X_P(q)$ 
    in rows $I = \{m+1-i, m+1\}$ and columns $J = \{1, m+1\}$.
    
    By the discussion preceding this theorem, the $2 \times 2$ minors correspond to sums over pairs of non-intersecting paths in the network $\Gamma_{x_n}$.
    By examining this network carefully, it is not difficult to identify which pair of paths corresponds to the minimal term. We will then verify that
    in the even and odd cases, this minimal term corresponds to the correct $q^N$ stated in the theorem.

    Consider first the case when $n$ is even. By Theorem \ref{thm:q_matrix_formula}, the first column of $X_P(q)$ has a factor of $q^m$, and from the form
    of the $\Lambda_m^-(a_n)$ matrix, we see that the last column of $X_P(q)$ is the first column of $X_{P'}(q)$. Therefore the numerator we wish to analyze 
    is equal to $q^{-m} \Delta$, where $\Delta$ is the $2 \times 2$ minor of $X_P(q)$ mentioned above.

    By the Lindstr\"{o}m-Gessel-Viennot formula, $\Delta = \sum \mathrm{wt}(p_1)\mathrm{wt}(p_2)$, where the sum is over all pairs of non-intersecting paths
    on the network $\Gamma_{x_n}$ which start at vertices $m+1-i$ and $m+1$ on theft, and end at vertices $1$ and $m+1$ on the right. To identify the minimal term
    of $\Delta$, we need to identify the pair of paths whose weights has the smallest power of $q$. An example is pictured in Figure \ref{fig:gamma_minimal_paths}.
    Clearly, the bottom of the two paths must travel horizontally 
    across the bottom strand, contributing a weight of $1$. Since the horizontal edges have weights $q^m$, $q^{m-1}$, $\dots$, $q$, $1$ going top-to-bottom,
    in order to minimize the weight we just need to find the path which starts at vertex $m+1-i$ and stays as close to the bottom of the network for as long as possible
    (without touching the bottom horizontal strand), and eventually ends at the top vertex labelled 1 on the right. The path which does this will go vertically down
    to the horizontal level $m$ (the second from the bottom), and then travel horizontally all the way until the last vertical strip, at which point it will go up 
    to the top and end at vertex 1. Along the way, on the low horizontal strip it will pick up a total weight of $q^{a_1+a_2+\cdots+a_n-1}$. On the very last
    horizontal edge (after going up to the top strand) it will get an extra weight of $q^m$. But as we mentioned above, this factor of $q^m$ will cancel.

    The case when $n$ is odd is similar, and we omit the details.
\end {proof}

\begin {figure}
\centering
\begin {tikzpicture}
    %vertices
    \foreach \x in {0,1,2,3,4,5,6,7,8,9} {
        \foreach \y in {0,1,2,3} {
            \draw[fill=black] (\x,\y) circle (0.06);
        }
    }

    % horizontal edges
    \foreach \x in {0,1,2,3,4,5,6,7,8} {
        \foreach \y in {0,1,2,3} {
            \draw[-latex] (\x,\y) -- (\x+0.94,\y);
        }
    }

    % vertical downward arrows
    \foreach \x in {1,2,4,5} {
        \foreach \y in {1,2,3} {
            \draw[-latex] (\x,\y) -- (\x,\y-0.94);
        }
    }

    % vertical upward arrows
    \foreach \x in {3,6,7,8} {
        \foreach \y in {0,1,2} {
            \draw[-latex] (\x,\y) -- (\x,\y+0.94);
        }
    }

    % edge labels
    \foreach \x in {1,2,3,4,5,6,7,8} {
        \draw (\x+0.6,1) node[above] {$q$};
        \draw (\x+0.6,2) node[above] {$q^2$};
        \draw (\x+0.6,3) node[above] {$q^3$};
    }

    % vertex labels
    \draw (-0.1,0) node[left] {$4$};
    \draw (-0.1,1) node[left] {$3$};
    \draw (-0.1,2) node[left] {$2$};
    \draw (-0.1,3) node[left] {$1$};
    \draw (9.1,0) node[right] {$4$};
    \draw (9.1,1) node[right] {$3$};
    \draw (9.1,2) node[right] {$2$};
    \draw (9.1,3) node[right] {$1$};

    % paths
    \draw[blue, line width=1.2] (0,0) -- (9,0);
    \draw[blue, line width=1.2] (0,2) -- (1,2) -- (1,1) -- (8,1) -- (8,3) -- (9,3);
\end {tikzpicture}
\caption {The pair of non-intersecting paths in $\Gamma_{x_n}$ from $I = \{2,4\}$ to $J=\{1,4\}$ with minimal weight. 
          The picture illustrates the special case $m=3$ and $x_n = [2,1,2,3]$.}
\label {fig:gamma_minimal_paths}
\end {figure}

\begin{remark}
Using the matrix formulation developed in the previous sections, we can give an alternate proof of \autoref{thm:stabilization}.  We focus on the case when $n$ is even, but the odd case is similar.  First, by expanding out the matrix multiplication, one can check that for any $(m+1)\times(m+1)$ matrix Y, each $2 \times 2$ minor of $R_m(q)Y$ and $L_m(q)Y$ is a positive power of $q$ times a sum of $2\times 2$ minors of $Y$.  Thus, by induction, we can see that the $2\times 2$ minors of $X_\calG(q) = R_m(q)^{a_1}L_m(q)^{a_2}\cdots R_m(q)^{a_{n-1}}L_m(q)^{a_n}$ are divisible by $q^{a_1+a_2+\cdots+a_n}$.  Moreover, if the $2\times 2$ minor is taken in the first and last columns, the induction yields that such minors are divisible by $q^{a_1+a_2+\cdots+a_n + m - 1}$.  Then interpreting the numerator of $q^m(r^q_{im}(x_n) - r^q_{im}(x_{n-1}))$ as a $2\times 2$ minor of $X_\calG(q)$ yields the desired stabilization.
\end{remark}

As in the $m=1$ case, the previous result can be strengthened to use any sequence converging to $x$.

\begin {theorem}
    Let $x = [a_1,a_2,\dots]$ be an irrational number, and let $y_1,y_2,y_3,\dots$ be any sequence of rationals
    converging to $x$. Then the series expansions of $r_{i,m}^q(y_n)$ stabilize to $r_{i,m}^q(x)$ as $n \to \infty$.
\end {theorem}
\begin {proof}
    Because the statement of Theorem \ref{thm:stabilization} does not depend on $m$, the proof of this result is essentially the 
    same as the proof given in \cite{mgo_22} (see Lemma 3.1 and Theorem 1). The idea is as follows. Let $x_n = [a_1,\dots,a_n]$
    be the rational convergents of $x$. We want to see that for any $k$, the sequence of $q^k$ coefficients in the power series expansions
    of $r_{i,m}^q(y_n)$ are eventually constant. Let $\ell$ be an even number such that $a_1+a_2+\cdots+a_\ell > k$. Because $y_n \to x$, there is some
    $N$ large enough so that for $n > N$ we have $x_{\ell-1} < y_n < x_\ell$. In terms of continued fractions this inequality is equivalent
    to saying that the continued fraction for $y_n$ begins with $[a_1,\dots,a_{\ell-1}, b,\dots]$ with $b > a_\ell$. By Theorem \ref{thm:stabilization},
    $r_{i,m}^q(y_n)$ and $r_{i,m}^q(x_{\ell-1})$ agree up to the $q^{a_1+\cdots+a_\ell-1}$ term. But we assumed that $a_1+\cdots+a_\ell > k$,
    and so all $y_n$'s with $n > N$ will have the same $q^k$ coefficient.
\end {proof}

\bigskip

\begin {example}
    In \cite{mosz_23} we showed that for the golden ratio $\varphi = \frac{1}{2}(1+\sqrt{5})$, we have
    \[ \mathrm{CF}_2(\varphi) = \left( r_{22}(\varphi), r_{12}(\varphi), r_{02}(\varphi)\right)^\top = \left( \sigma, \rho, \, 1\right)^\top \]
    where $\sigma = 4\cos(\pi/7)^2-1$ and $\rho=2\cos(\pi/7)$, which are the lengths of the diagonals in a regular $7$-gon (with side length 1).
    If $f_n$ are the Fibonacci numbers, then the ratios $x_n = f_n/f_{n-1}$ converge to $\varphi$. These rationals
    have continued fractions $x_n = [1,1,\dots,1,1]$, so we can compute $r_{im}(x_n)$ using matrix products $\Lambda_m(1)^n$.
    By Theorem \ref{thm:stabilization}, the series expansions for $r^q_{i2}(x_n)$ stabilize. The first several terms are given below:
    \begin {align*} 
        r^q_{22}(\varphi) &= 1 + q^2 + q^4 - q^5 + q^7 - 3q^8 + 6q^9 - 10q^{10} \\
                          & \quad\;\;\; +17q^{11} - 24q^{12} + 25q^{13} - 15q^{14} - 21q^{15} + 107q^{16} + \cdots \\
        r^q_{12}(\varphi) &= 1 + q^2 - q^6 + 2q^7 - 4q^8 + 7q^9 - 9q^{10} \\
                          & \quad\;\;\; + 11q^{11} - 11q^{12} + 2q^{13} + 22q^{14} - 74q^{15} + 171q^{16} + \cdots
    \end {align*}
\end {example}

\section{The Positivity Phenomenon} \label{sec:positivity}
Another important property of Morier-Genoud and Ovsienko's $q$-rational numbers is the \emph{total positivity}. If $\frac{r}{s} > \frac{a}{b}$ (equivalently $br-as > 0$), then
the numerators and denominators of the $q$-rationals satisfy the property that $\mathcal{B}(q)\mathcal{R}(q) - \mathcal{A}(q)\mathcal{S}(q)$ 
has all positive coefficients. We prove the following analogous property for the higher $q$-continued fractions.

\begin {theorem} \label{thm:positivity}
    Suppose $\frac{r}{s} > \frac{a}{b}$, with $r_{im}^q \left( \frac{r}{s} \right) = \frac{\mathcal{R}(q)}{\mathcal{S}(q)}$ and $r_{im}^q \left( \frac{a}{b}\right) = \frac{\mathcal{A}(q)}{\mathcal{B}(q)}$. 
    Then $\mathcal{B}(q)\mathcal{R}(q) - \mathcal{A}(q)\mathcal{S}(q)$ has all positive coefficients. 
\end {theorem}
\begin {proof}
    Let us use the following shorthand notations for certain $\Omega_m^{ij}$ sets: 
    \[ 
        R = \Omega_m^{m+1-i,1}\left(\mathcal{G} \left(\frac{r}{s}\right)\right), \quad \quad S = \Omega_m^{m+1,1}\left(\mathcal{G}\left(\frac{r}{s}\right)\right), 
    \]
    \[
        A = \Omega_m^{m+1-i,1}\left(\mathcal{G}\left(\frac{a}{b}\right)\right),  \quad \quad B = \Omega_m^{m+1,1}\left(\mathcal{G}\left(\frac{a}{b}\right)\right). 
    \]
    Then as mentioned in Remark \ref{rmk:quotient_formula}, $\mathcal{R}(q)$, $\mathcal{S}(q)$, $\mathcal{A}(q)$, $\mathcal{B}(q)$ are the generating functions
    for the sets $R$, $S$, $A$, $B$, respectively.
    Therefore $\mathcal{R}(q)\mathcal{B}(q)$
    is the generating function for the set $R \times B$. That is,
    \[ \mathcal{R}(q)\mathcal{B}(q) = \sum_{(\rho,\beta)} q^{\mathrm{wt}(\rho) + \mathrm{wt}(\beta)}, \]
    where $\rho$ is a $P$-partition on $\mathcal{G}\left(\frac{r}{s}\right)$ whose entries in the first column are at most $i$, and $\beta$ is a $P$-partition
    on $\mathcal{G}\left(\frac{a}{b}\right)$ whose values in the first column are all $0$.
    Similarly, $\mathcal{A}(q)\mathcal{S}(q)$ is the generating function for the set $S \times A$;
    i.e. pairs $(\sigma,\alpha)$, where $\alpha$ is a $P$-partition on $\mathcal{G}\left(\frac{a}{b}\right)$ whose values in the first column are at most $i$,
    and $\sigma$ is a $P$-partition on $\mathcal{G}\left(\frac{r}{s}\right)$ whose values in the first column are all $0$.

    The strategy of the proof will be to give a weight-preserving injection
    \begin {align*}
        \varphi \colon S \times A &\longrightarrow R \times B \\
        (\sigma,\alpha) &\mapsto (\rho,\beta)
    \end {align*}
    From this, it will follow that $\mathcal{R}(q)\mathcal{B}(q) - \mathcal{S}(q)\mathcal{A}(q)$ has positive coefficients.

    Suppose the continued fractions for $\frac{r}{s}$ and $\frac{a}{b}$ agree for the first $k$ terms; that is, $\frac{r}{s} = [c_1,c_2,\dots,c_k,a_{k+1},\dots]$
    and $\frac{a}{b} = [c_1,c_2,\dots,c_k,b_{k+1},\dots]$, with $a_{k+1} \neq b_{k+1}$. We assume that $\frac{r}{s} > \frac{a}{b}$.
    If $k$ is even, this means $a_{k+1} > b_{k+1}$ (and if $k$ is odd, then $a_{k+1} < b_{k+1}$). We will just consider the case when $k$ is even (the case for $k$ odd is similar).

    The first $k$ straight segments (consisting of the first $c_1+c_2+\cdots+c_k$ squares)
    are the same in $\mathcal{G}\left(\frac{r}{s}\right)$ and $\mathcal{G}\left(\frac{a}{b}\right)$. 
    The first place where they differ is in the $(k+1)^\mathrm{st}$ segment, where $\mathcal{G}\left(\frac{r}{s}\right)$ is longer than $\mathcal{G}\left(\frac{a}{b}\right)$.
    This observation will be used to define the injective map $\varphi$ mentioned above. 

    Let $(\sigma, \alpha) \in S \times A$, corresponding to one of the terms in $\mathcal{S}(q)\mathcal{A}(q)$.
    Recall from the definitions of $\Omega_m^{i,j}$ that $\alpha$ has values at most $i$ in the first column of $\mathcal{G}\left(\frac{a}{b}\right)$
    and $\sigma$ has value $0$ in all boxes in the first column of $\mathcal{G}\left(\frac{r}{s}\right)$. 

    Let $d = c_1+c_2+\cdots+c_k+b_{k+1}$, that is, $d$ is the position of the last square in the initial segment that $\mathcal{G}\left(\frac{r}{s}\right)$ and $\mathcal{G}\left(\frac{a}{b}\right)$ have in common.  We say that a position $j \in \{1,2,\dots,d\}$ is \emph{swappable} in the pair $(\sigma,\alpha)$ if swapping the labels in all squares up to (and including) the $j^\mathrm{th}$ one in $\sigma$ and $\alpha$ yields a pair of $P$-partitions.  Any pair $(\sigma,\alpha) \in S \times A$ has a swappable position.  In particular, let $k \in \{1,\dots,d-1\}$ be minimal such that $\sigma(k + 1) \geq \alpha(k + 1)$, or, if no such position exists, let $k = d$.  It is then straightforward to check that $k$ is a swappable position.  
    
     Then we define $\varphi(\sigma, \alpha) := (\rho, \beta)$, obtained by swapping the labels on all squares up to (and including) the first swappable position.  That is, letting $\ell$ denote the first swappable position, we set $\rho(j) = \alpha(j)$ and $\beta(j) = \sigma(j)$ for all $j \leq \ell$; and if $j > \ell$, then $\rho(j) = \sigma(j)$ and $\beta(j) = \alpha(j)$. See Figure \ref{fig:swap} for an illustration.

     This map is weight-preserving, since the sum of the parts of $\alpha$ and $\sigma$ are the same as the sum of the parts of $\rho$ and $\beta$. Moreover, since the initial length $\ell$ portions of the $P$-partitions are not changed, then the first swappable position in the resulting pair $(\rho,\beta)$ is the same as that of $(\sigma,\alpha)$.  Hence we can similarly construct the preimage $(\sigma,\alpha)$ from $(\rho,\beta)$ by swapping the labels up to the first swappable position in $(\rho,\beta)$.  Therefore, we can conclude that the map $\varphi$ is injective.  
\end {proof}

\begin{figure}[h]
\centering
\begin{tikzpicture}[scale = 0.7]
	\begin{pgfonlayer}{nodelayer}
		\node [style=none] (0) at (-11.5, -2) {};
		\node [style=none] (1) at (-10.5, -2) {};
		\node [style=none] (3) at (-11.5, 2) {};
		\node [style=none] (6) at (-8.5, 3) {};
		\node [style=none] (7) at (-9.5, 3) {};
		\node [style=none] (8) at (-6.5, -2) {};
		\node [style=none] (9) at (-5.5, -2) {};
		\node [style=none] (10) at (-5.5, 1) {};
		\node [style=none] (11) at (-6.5, 2) {};
		\node [style=none] (12) at (-4.5, 2) {};
		\node [style=none] (13) at (-3.5, 1) {};
		\node [style=none] (14) at (-3.5, 3) {};
		\node [style=none] (15) at (-4.5, 3) {};
		\node [style=none] (16) at (1.5, -2) {};
		\node [style=none] (17) at (2.5, -2) {};
		\node [style=none] (18) at (2.5, 1) {};
		\node [style=none] (19) at (1.5, 2) {};
		\node [style=none] (20) at (3.5, 2) {};
		\node [style=none] (21) at (4.5, 1) {};
		\node [style=none] (22) at (4.5, 3) {};
		\node [style=none] (23) at (3.5, 3) {};
		\node [style=none] (24) at (6.5, -2) {};
		\node [style=none] (25) at (7.5, -2) {};
		\node [style=none] (26) at (7.5, 1) {};
		\node [style=none] (27) at (6.5, 2) {};
		\node [style=none] (28) at (8.5, 2) {};
		\node [style=none] (29) at (9.5, 1) {};
		\node [style=none] (30) at (9.5, 3) {};
		\node [style=none] (31) at (8.5, 3) {};
		\node [style=none] (33) at (-11.5, 1) {};
		\node [style=none] (34) at (-11.5, 0) {};
		\node [style=none] (35) at (-10.5, 0) {};
		\node [style=none] (36) at (-10.5, -1) {};
		\node [style=none] (37) at (-11.5, -1) {};
		\node [style=none] (38) at (-9.5, 1) {};
		\node [style=none] (39) at (-8.5, 2) {};
		\node [style=none] (44) at (-7.5, 5) {};
		\node [style=none] (50) at (-9.5, 4) {};
		\node [style=none] (51) at (-2.5, 3) {};
		\node [style=none] (52) at (-2.5, 2) {};
		\node [style=none] (53) at (-1.5, 3) {};
		\node [style=none] (54) at (-1.5, 2) {};
		\node [style=none] (55) at (-1.5, 4) {};
		\node [style=none] (56) at (-2.5, 4) {};
		\node [style=none] (57) at (-3.5, 2) {};
		\node [style=none] (58) at (-6.5, 1) {};
		\node [style=none] (59) at (-6.5, 0) {};
		\node [style=none] (60) at (-5.5, 0) {};
		\node [style=none] (61) at (-5.5, -1) {};
		\node [style=none] (62) at (-6.5, -1) {};
		\node [style=none] (63) at (-5.5, 2) {};
		\node [style=none] (64) at (-4.5, 1) {};
		\node [style=none] (65) at (3.5, 5) {};
		\node [style=none] (66) at (4.5, 5) {};
		\node [style=none] (67) at (4.5, 6) {};
		\node [style=none] (68) at (5.5, 6) {};
		\node [style=none] (69) at (5.5, 4) {};
		\node [style=none] (70) at (4.5, 4) {};
		\node [style=none] (71) at (11.5, 2) {};
		\node [style=none] (72) at (10.5, 3) {};
		\node [style=none] (73) at (11.5, 4) {};
		\node [style=none] (74) at (10.5, 4) {};
		\node [style=none] (75) at (9.5, 2) {};
		\node [style=none] (76) at (7.5, 2) {};
		\node [style=none] (77) at (5.5, 5) {};
		\node [style=none] (78) at (3.5, 4) {};
		\node [style=none] (79) at (4.5, 2) {};
		\node [style=none] (80) at (3.5, 1) {};
		\node [style=none] (81) at (2.5, 2) {};
		\node [style=none] (82) at (1.5, 1) {};
		\node [style=none] (83) at (1.5, 0) {};
		\node [style=none] (84) at (2.5, 0) {};
		\node [style=none] (85) at (2.5, -1) {};
		\node [style=none] (86) at (1.5, -1) {};
		\node [style=none] (87) at (6.5, 1) {};
		\node [style=none] (88) at (6.5, 0) {};
		\node [style=none] (89) at (7.5, 0) {};
		\node [style=none] (90) at (7.5, -1) {};
		\node [style=none] (91) at (6.5, -1) {};
		\node [style=none] (92) at (10.5, 2) {};
		\node [style=none] (93) at (11.5, 3) {};
		\node [style=none] (94) at (8.5, 1) {};
	\node [style=none] (95) at (-11, 1.5) {$0$};
		\node [style=none] (96) at (-11, 0.5) {$0$};
		\node [style=none] (97) at (-11, -0.5) {$0$};
		\node [style=none] (98) at (-11, -1.5) {$0$};
		\node [style=none] (99) at (7, 1.5) {$0$};
		\node [style=none] (100) at (7, 0.5) {$0$};
		\node [style=none] (101) at (7, -0.5) {$0$};
		\node [style=none] (102) at (7, -1.5) {$0$};
		\node [style=none] (103) at (-10, 1.5) {$1$};
		\node [style=none] (104) at (-9, 1.5) {$2$};
		\node [style=none] (105) at (-9, 2.5) {$2$};
		\node [style=none] (106) at (2, 1.5) {$2$};
		\node [style=none] (107) at (2, 0.5) {$2$};
		\node [style=none] (108) at (2, -0.5) {$3$};
		\node [style=none] (109) at (2, -1.5) {$3$};
		\node [style=none] (110) at (-6, 1.5) {$2$};
		\node [style=none] (111) at (-6, 0.5) {$2$};
		\node [style=none] (112) at (-6, -0.5) {$3$};
		\node [style=none] (113) at (-6, -1.5) {$3$};
		\node [style=none] (114) at (-4, 2.5) {$0$};
		\node [style=none] (115) at (-4, 1.5) {$3$};
		\node [style=none] (116) at (-5, 1.5) {$2$};
		\node [style=none] (117) at (3, 1.5) {$2$};
		\node [style=none] (118) at (4, 1.5) {$2$};
		\node [style=none] (119) at (4, 2.5) {$2$};
		\node [style=none] (120) at (8, 1.5) {$1$};
		\node [style=none] (121) at (9, 1.5) {$3$};
		\node [style=none] (122) at (9, 2.5) {$0$};
		\node [style=none] (123) at (-3, 2.5) {$1$};
		\node [style=none] (124) at (-2, 2.5) {$2$};
		\node [style=none] (125) at (-2, 3.5) {$1$};
		\node [style=none] (126) at (10, 2.5) {$1$};
		\node [style=none] (127) at (11, 2.5) {$2$};
		\node [style=none] (128) at (11, 3.5) {$1$};
		\node [style=none] (129) at (-8, 5.5) {$1$};
		\node [style=none] (130) at (-8, 4.5) {$2$};
		\node [style=none] (131) at (-9, 3.5) {$1$};
		\node [style=none] (132) at (-9, 4.5) {$1$};
		\node [style=none] (133) at (5, 5.5) {$1$};
		\node [style=none] (134) at (5, 4.5) {$2$};
		\node [style=none] (135) at (4, 3.5) {$1$};
		\node [style=none] (136) at (4, 4.5) {$1$};
		\node [style=none] (137) at (-9, -3) {$\sigma$};
		\node [style=none] (138) at (-4, -3) {$\alpha$};
		\node [style=none] (139) at (4, -3) {$\rho$};
		\node [style=none] (140) at (9, -3) {$\beta$};
		\node [style=none] (141) at (-1, 1.5) {};
		\node [style=none] (142) at (1, 1.5) {};
		\node [style=none] (143) at (-8.5, 6) {};
		\node [style=none] (144) at (-7.5, 6) {};
		\node [style=none] (145) at (-7.5, 4) {};
		\node [style=none] (146) at (-8.5, 4) {};
		\node [style=none] (147) at (-8.5, 1) {};
		\node [style=none] (148) at (-10.5, 1) {};
		\node [style=none] (149) at (-10.5, 2) {};
		\node [style=none] (150) at (-9.5, 2) {};
		\node [style=none] (151) at (-9.5, 5) {};
		\node [style=none] (152) at (-8.5, 5) {};
		\node [style=none] (154) at (-11.5, 2) {};
		\node [style=none] (155) at (-10.5, 2) {};
		\node [style=none] (156) at (-10.5, -2) {};
		\node [style=none] (157) at (-11.5, -2) {};
		\node [style=none] (158) at (-9.5, 3) {};
		\node [style=none] (159) at (-8.5, 3) {};
		\node [style=none] (160) at (-8.5, 1) {};
		\node [style=none] (161) at (-10.5, 1) {};
		\node [style=none] (162) at (-10.5, -2) {};
		\node [style=none] (163) at (-11.5, -2) {};
		\node [style=none] (164) at (-11.5, 2) {};
		\node [style=none] (166) at (-9.5, 2) {};
		\node [style=none] (167) at (0, 2) {$\varphi$};
	\end{pgfonlayer}
	\begin{pgfonlayer}{edgelayer}
		\draw [style=shade] (12.center)
			 to (11.center)
			 to (8.center)
			 to (9.center)
			 to (10.center)
			 to (13.center)
			 to (14.center)
			 to (15.center);
		\draw [style=shade] (20.center)
			 to (19.center)
			 to (16.center)
			 to (17.center)
			 to (18.center)
			 to (21.center)
			 to (22.center)
			 to (23.center);
		\draw [style=shade] (28.center)
			 to (27.center)
			 to (24.center)
			 to (25.center)
			 to (26.center)
			 to (29.center)
			 to (30.center)
			 to (31.center);
    	\draw [style=shade] (158.center)
			 to (159.center)
			 to (160.center)
			 to (161.center)
			 to (162.center)
			 to (163.center)
			 to (164.center)
			 to (166.center)
			 to cycle;
		\draw [style=fade] (37.center) to (36.center);
		\draw [style=fade] (34.center) to (35.center);
		\draw [style=thick] (7.center) to (6.center);
		\draw [style=thick] (51.center)
			 to (15.center)
			 to (12.center)
			 to (11.center)
			 to (8.center)
			 to (9.center)
			 to (10.center)
			 to (13.center)
			 to (57.center)
			 to (54.center)
			 to (55.center)
			 to (56.center)
			 to cycle;
		\draw [style=thick] (62.center) to (61.center);
		\draw [style=thick] (60.center) to (59.center);
		\draw [style=thick] (58.center) to (10.center);
		\draw [style=thick] (10.center) to (63.center);
		\draw [style=thick] (12.center) to (64.center);
		\draw [style=thick] (12.center) to (57.center);
		\draw [style=thick] (57.center) to (14.center);
		\draw [style=thick] (51.center) to (52.center);
		\draw [style=thick] (51.center) to (53.center);
		\draw [style=thick] (66.center)
			 to (65.center)
			 to (20.center)
			 to (19.center)
			 to (16.center)
			 to (17.center)
			 to (18.center)
			 to (21.center)
			 to (70.center)
			 to (69.center)
			 to (68.center)
			 to (67.center)
			 to cycle;
		\draw [style=thick] (31.center)
			 to (28.center)
			 to (76.center)
			 to (26.center)
			 to (29.center)
			 to (75.center)
			 to (71.center)
			 to (73.center)
			 to (74.center)
			 to (72.center)
			 to cycle;
		\draw [style=thick] (66.center) to (77.center);
		\draw [style=thick] (70.center) to (66.center);
		\draw [style=thick] (70.center) to (78.center);
		\draw [style=thick] (23.center) to (22.center);
		\draw [style=thick] (79.center) to (20.center);
		\draw [style=thick] (20.center) to (80.center);
		\draw [style=thick] (81.center) to (18.center);
		\draw [style=thick] (18.center) to (82.center);
		\draw [style=thick] (83.center) to (84.center);
		\draw [style=thick] (86.center) to (85.center);
		\draw [style=thick] (72.center) to (93.center);
		\draw [style=thick] (72.center) to (92.center);
		\draw [style=thick] (30.center) to (75.center);
		\draw [style=thick] (75.center) to (28.center);
		\draw [style=thick] (28.center) to (94.center);
		\draw [style=fade] (24.center)
			 to (25.center)
			 to (76.center)
			 to (27.center)
			 to cycle;
		\draw [style=fade] (90.center) to (91.center);
		\draw [style=fade] (88.center) to (89.center);
		\draw [style=fade] (87.center) to (26.center);
		\draw [style=red arrow thick] (141.center) to (142.center);
		\draw [style=thick] (150.center)
			 to (149.center)
			 to (148.center)
			 to (147.center)
			 to (146.center)
			 to (145.center)
			 to (144.center)
			 to (143.center)
			 to (152.center)
			 to (151.center)
			 to cycle;
		\draw [style=fade] (157.center)
			 to (156.center)
			 to (155.center)
			 to (154.center)
			 to cycle;
		\draw [style=fade] (33.center) to (148.center);
		\draw [style=thick] (150.center) to (38.center);
		\draw [style=thick] (150.center) to (39.center);
		\draw [style=thick] (50.center) to (146.center);
		\draw [style=thick] (146.center) to (152.center);
		\draw [style=thick] (152.center) to (44.center);
	\end{pgfonlayer}
\end{tikzpicture}
\caption {An example of the map $\varphi$ from the proof of Theorem \ref{thm:positivity}. 
    The shaded region indicates the common part of $\mathcal{G}\left(\frac{r}{s}\right)$ and $\mathcal{G}\left(\frac{a}{b}\right)$. 
    The portion with thinner outlines are the entries where the $P$-partitions must be $0$.
    In this example, the map $\varphi$ swaps the labels in the first $5$ boxes.}
\label {fig:swap}
\end{figure}

\bigskip

\begin {example}\label{ex: difference}
    We have $\frac{5}{2} > \frac{7}{3}$, with 
    \begin {align*} 
        r^q_{22} \left( \frac{5}{2} \right) &= \frac{1+2q+3q^2+3q^3+3q^4+q^5+q^6}{1+q+q^2} \\
        r^q_{22} \left( \frac{7}{3} \right) &= \frac{1 + 2q + 4q^2 + 4q^3 + 5q^4 + 4q^5 + 3q^6 + q^7 + q^8}{1+q+2q^2+q^3+q^4} 
    \end {align*}
    The corresponding difference is
    \[ \mathcal{R}(q)\mathcal{B}(q) - \mathcal{A}(q)\mathcal{S}(q) = q^3 + 2q^4 + 2q^5 + 2q^6 + q^7 + q^8 \]

\end{example}

\bigskip

\begin{remark}
    Not only are the coefficients of $\mathcal{R}(q)\mathcal{B}(q) - \mathcal{A}(q)\mathcal{S}(q)$ positive integers, but the proof of Theorem \ref{thm:positivity}
    gives some combinatorial interpretation as well. This difference is the generating function for pairs of $P$-partitions which are in the \emph{complement}
    of the image of the map $\varphi$. It follows from the proof of Theorem~\ref{thm:positivity} that this complement consists of the pairs in $R \times B$ such that no position in $\{1,2,\dots,d\}$ is swappable.  These are precisely the pairs $(\rho,\beta) \in R \times B$ such that 
    \begin{itemize}
        \item for all $1 \leq j < d$, we have that both ${\max(\rho(j),\beta(j)) > \min(\rho(j+1),\beta(j+1))}$ and ${\min(\rho(j),\beta(j)) < \max(\rho(j+1),\beta(j+1))}$ hold, and
        \item at least one of $\rho(d) > \beta(d+1)$ or $\beta(d) < \rho(d+1)$ holds.
    \end{itemize}.
\end {remark}

\begin{example}
    If we take $\frac{r}{s} = \frac{5}{2}$ and $\frac{a}{b} = \frac{7}{3}$ as in Example \ref{ex: difference}, then there are $9$ pairs of $P$-partitions in the \emph{complement} of the image of the map $\varphi$.  These pairs are shown in Figure~\ref{fig: difference p partitions}.
    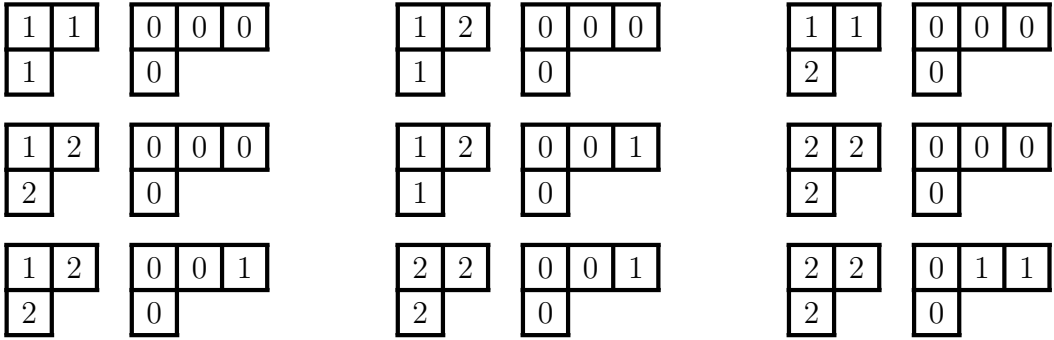
\begin{figure}[ht]
\centering
    \begin{tikzpicture}[scale = 0.8]
    % Define a macro for your TikZ picture
    \newcommand{\mytikzpicture}[7]{
        \begin{tikzpicture}[scale = 0.6]
            \begin{pgfonlayer}{nodelayer}
                \node [style=none] (0) at (0,2) {};
                \node [style=none] (1) at (1,2) {};
                \node [style=none] (2) at (2,2) {};
                \node [style=none] (3) at (0,1) {};
                \node [style=none] (4) at (1,1) {};
                \node [style=none] (5) at (2,1) {};
                \node [style=none] (6) at (0,0) {};
                \node [style=none] (7) at (1,0) {};
                \node [style=none] at (0.5,0.5) {#1};
                \node [style=none] at (1.5,1.5) {#3};
                \node [style=none] at (0.5,1.5) {#2};
            \end{pgfonlayer}
            \begin{pgfonlayer}{edgelayer}
                \draw [style=thick] (2.east) to (0.west);
                \draw [style=thick] (3.west) to (5.east);
                \draw [style=thick] (6.west) to (7.east);
                \draw [style=thick] (6.center) to (0.center);
                \draw [style=thick] (1.center) to (7.center);
                \draw [style=thick] (2.center) to (5.center);
            \end{pgfonlayer}
        \end{tikzpicture}
        \,
        \begin{tikzpicture}[scale = 0.6]
        	\begin{pgfonlayer}{nodelayer}
        		\node [style=none] (0) at (0,2) {};
        		\node [style=none] (1) at (1,2) {};
                \node [style=none] (2) at (2,2) {};
                \node [style=none] (3) at (0,1) {};
        		\node [style=none] (4) at (1,1) {};
                \node [style=none] (5) at (2,1) {};
                \node [style=none] (6) at (0,0) {};
        		\node [style=none] (7) at (1,0) {};
                \node [style=none] (8) at (3,1) {};
                \node [style=none] (9) at (3,2) {};
                \node [style=none] at (0.5,0.5) {#4};
                \node [style=none] at (1.5,1.5) {#6};
                \node [style=none] at (0.5,1.5) {#5};
                \node [style=none] at (2.5,1.5) {#7};
        	\end{pgfonlayer}
        	\begin{pgfonlayer}{edgelayer}
        		\draw [style=thick] (9.east) to (0.west);
        		\draw [style=thick] (3.west) to (8.east);
        		\draw [style=thick] (6.west) to (7.east);
                \draw [style=thick] (6.center) to (0.center);
        		\draw [style=thick] (1.center) to (7.center);
        		\draw [style=thick] (2.center) to (5.center);
                \draw [style=thick] (8.center) to (9.center);
        	\end{pgfonlayer}
        \end{tikzpicture}
    }
    \node at (0,0) {\mytikzpicture{1}{1}{1}{0}{0}{0}{0}};
    \node at (0,-2) {\mytikzpicture{2}{1}{2}{0}{0}{0}{0}};
    \node at (0,-4) {\mytikzpicture{2}{1}{2}{0}{0}{0}{1}};
    \node at (6.5,0) {\mytikzpicture{1}{1}{2}{0}{0}{0}{0}};
    \node at (6.5,-2) {\mytikzpicture{1}{1}{2}{0}{0}{0}{1}};
    \node at (6.5,-4) {\mytikzpicture{2}{2}{2}{0}{0}{0}{1}};
    \node at (13,0) {\mytikzpicture{2}{1}{1}{0}{0}{0}{0}};
    \node at (13,-2) {\mytikzpicture{2}{2}{2}{0}{0}{0}{0}};
    \node at (13,-4) {\mytikzpicture{2}{2}{2}{0}{0}{1}{1}};
\end{tikzpicture}
\caption{The pairs of $P$-partitions that give a combinatorial interpretation for $\mathcal{R}(q)\mathcal{B}(q) - \mathcal{A}(q)\mathcal{S}(q)$ in Example \ref{ex: difference}.}
\label{fig: difference p partitions}
\end{figure}
\end{example}

\section{Acknowledgements}
The authors would like to thank Gregg Musiker for helpful conversations and for his early contributions to this project.  Amanda Burcroff is grateful to her advisor, Lauren Williams, for support throughout this project.

\bibliographystyle{alpha}
\bibliography{refs}

\end{document}

%% file: tikzsetup.tex
\usepackage{tikz}
\usetikzlibrary{backgrounds}
\usetikzlibrary{arrows}
\usetikzlibrary{shapes,shapes.geometric,shapes.misc}
\tikzstyle{tikzfig}=[baseline=-0.25em,scale=0.5]
\pgfkeys{/tikz/tikzit fill/.initial=0}
\pgfkeys{/tikz/tikzit draw/.initial=0}
\pgfkeys{/tikz/tikzit shape/.initial=0}
\pgfkeys{/tikz/tikzit category/.initial=0}
\pgfdeclarelayer{edgelayer}
\pgfdeclarelayer{nodelayer}
\pgfsetlayers{background,edgelayer,nodelayer,main}

\tikzstyle{none}=[inner sep=0mm]

\tikzstyle{label}=[inner sep=0.15mm, font={\footnotesize}]
\tikzstyle{vert}=[fill=black, draw=white, shape=circle, line width=0.25mm, tikzit shape=circle, inner sep=0.35mm]
\tikzstyle{label-small}=[inner sep=0.1mm, font={\scriptsize}]
\tikzstyle{label-white}=[fill=white, draw=none, shape=circle, inner sep=0.07mm, font={\scriptsize}]
\tikzstyle{vert-white}=[fill=white, draw=black, shape=circle, inner sep=0.35 mm]
\tikzstyle{vert-red}=[fill=red, draw=white, shape=circle, inner sep=0.4 mm, font={\scriptsize}, text=black]
\tikzstyle{Vert}=[fill=black, draw=black, shape=circle, minimum size=0.4em, inner sep=0.4pt, scale=0.6]
\tikzstyle{label-red}=[fill=white, draw=black, shape=circle, scale=0.4, color=red]
\tikzstyle{label small}=[fill=none, draw=none, shape=circle, scale=0.7, inner sep = 0.1]
\tikzstyle{bullet}=[fill=black, draw=black, shape=circle, scale=0.5]
\tikzstyle{G-vert}=[fill=white, draw=black, shape=circle, inner sep=0.7pt, minimum size=0.4em, scale=0.6]
\tikzstyle{label}=[inner sep=0.15mm, font={\footnotesize}]
\tikzstyle{vert}=[fill=black, draw=white, shape=circle, line width=0.15mm, tikzit shape=circle, inner sep=1 mm, scale = 1.5]
\tikzstyle{nodes}=[fill=white, draw=none, shape=circle, inner sep=0.7pt, font={\scriptsize}, minimum size=11pt]
\tikzstyle{label-s}=[inner sep=0.1mm, font={\scriptsize}]

\tikzstyle{label-w}=[fill=white, draw=white, shape=circle, inner sep=0.07mm, font={\scriptsize}]
\tikzstyle{charge}=[fill=white, draw={rgb,255: red,16; green,77; blue,209}, shape=circle, inner sep=0.1 em, scale=0.8]

\tikzstyle{label}=[inner sep=0.15mm, font={\footnotesize}]
\tikzstyle{vert}=[fill=black, draw=white, shape=circle, line width=0.3mm, tikzit shape=circle, inner sep=0.35mm,minimum size=6pt]
\tikzstyle{label-s}=[inner sep=0.1mm, font={\scriptsize}]
\tikzstyle{label-w}=[fill=white, draw=none, shape=circle, inner sep=0.07mm, font={\scriptsize}]
\tikzstyle{charge}=[fill=white, draw={rgb,255: red,16; green,77; blue,209}, shape=circle, inner sep=0.2 em, scale=0.6]
\tikzstyle{particle}=[fill=black, draw=black, shape=circle, inner sep=0.76 mm]
\tikzstyle{vert-white}=[fill = white, draw = black, shape = circle, inner sep = 1.89 pt]

\tikzstyle{blue thick}=[-, tikzit draw=blue, draw=blue, line width=1.7pt]
\tikzstyle{thick}=[-, ultra thick]
\tikzstyle{dot}=[-, dotted]
\tikzstyle{fade}=[-, opacity=0.5]
\tikzstyle{yellow}=[-, draw={rgb,255: red,208; green,208; blue,42}, line width=1.5]
\tikzstyle{arrow}=[->, line width=2em]
\tikzstyle{faded}=[-, opacity=0.5]
\tikzstyle{matching}=[-, draw=blue, line width=2.5pt, opacity=0.5]
\tikzstyle{orange}=[-, draw={rgb,255: red,255; green,128; blue,0}]
\tikzstyle{red arrow}=[draw=red, ->]
\tikzstyle{double edge}=[-, double, double distance=1 pt, draw=blue, line width=1.3, opacity=1]
\tikzstyle{single edge}=[-, opacity=1, draw=blue, line width=1.3]
\tikzstyle{dot}=[-, dotted]
\tikzstyle{red arrow thick}=[->, line width=0.8pt, draw=red]
\tikzstyle{dimer}=[-, line width=1.75 pt, draw=red]

\tikzstyle{dashed}=[-, densely dotted]
\tikzstyle{virtual}=[-, double]
\tikzstyle{RED}=[-, draw=red]
\tikzstyle{CYAN}=[-, draw=cyan]
\tikzstyle{BLUE}=[-, draw=blue]
\tikzstyle{LGREEN}=[-, draw=green]
\tikzstyle{DGREEN}=[-, draw={rgb,255: red,0; green,128; blue,128}]
\tikzstyle{PURPLE}=[-, draw={rgb,255: red,128; green,0; blue,128}]
\tikzstyle{ORANGE}=[-, draw={rgb,255: red,255; green,128; blue,0}]
\tikzstyle{MAGENTA}=[-, draw=magenta]
\tikzstyle{BLUE(dashed)}=[-, draw=blue, densely dotted]
\tikzstyle{GREEN}=[-, draw={rgb,255: red,0; green,208; blue,6}]
\tikzstyle{MAGENTA(dashed)}=[-, draw=magenta, densely dotted]
\tikzstyle{ORANGE(dashed)}=[-, draw={rgb,255: red,255; green,128; blue,0}, densely dotted]
\tikzstyle{RED(dashed)}=[-, draw={rgb,255: red,191; green,0; blue,64}, densely dotted]
\tikzstyle{PURPLE(dashed)}=[-, draw={rgb,255: red,128; green,0; blue,128}, densely dotted]
\tikzstyle{GREEN(dashed)}=[-, draw={rgb,255: red,0; green,208; blue,6}, densely dotted]
\tikzstyle{shade}=[-, opacity=0.4, draw={rgb,255: red,128; green,128; blue,128}, line width=6.5, fill=none, line cap=round,rounded corners]
\tikzstyle{hyper}=[-, fill={rgb,255: red,48; green,255; blue,214}, fill opacity=0.6, draw={rgb,255: red,18; green,229; blue,85}, tikzit fill={rgb,255: red,48; green,255; blue,214}]
\tikzstyle{hyper1}=[-, fill={rgb,255: red,230; green,138; blue,9}, draw={rgb,255: red,255; green,128; blue,0}, fill opacity=0.5]
\tikzstyle{hyper2}=[-, fill={rgb,255: red,128; green,179; blue,255}, draw={rgb,255: red,46; green,87; blue,115}, fill opacity=0.55]
\tikzstyle{new edge style 0}=[-, fill={rgb,255: red,245; green,255; blue,39}, draw={rgb,255: red,168; green,170; blue,22}, fill opacity=0.6]

\tikzstyle{blue thick}=[-, tikzit draw=blue, draw=blue, line width=1.7pt]
\tikzstyle{thick}=[-, ultra thick]
\tikzstyle{dot}=[, dotted]
\tikzstyle{fade}=[-, opacity=0.5]
\tikzstyle{shade}=[-, opacity=0.23, fill={rgb,255: red,191; green,191; blue,191}, draw=none]

\tikzset{%
every path/.append style={line width = 0.8 pt}
}